\documentclass[12pt]{amsart}
\usepackage[T1]{fontenc}
\usepackage[a4paper,top=4cm, bottom=5cm, left=2.75cm, right=2.75cm]{geometry}
\usepackage[utf8]{inputenc}
\usepackage[dvipsnames]{xcolor}
\usepackage{amsthm,amsmath,bm} 
\usepackage{mathtools}
\usepackage{latexsym,amssymb}
\usepackage{stmaryrd}
\usepackage[shortlabels]{enumitem}
\usepackage[hidelinks]{hyperref}
\usepackage{orcidlink}

\renewcommand{\P}{\mathbb{P}}
\newcommand{\N}{\mathbb{N}}
\newcommand{\Z}{\mathbb{Z}}

\newcommand{\R}{\mathbb{R}}

\newcommand{\mA}{\mathcal{A}}
\newcommand{\uA}{{\underline{\mathcal{A}}}}
\newcommand{\mB}{\mathcal{B}}
\newcommand{\mC}{\mathcal{C}}
\newcommand{\mE}{\mathcal{E}}
\newcommand{\mF}{\mathcal{F}}
\newcommand{\mH}{\mathcal{H}}
\newcommand{\mS}{\mathcal{S}}
\newcommand{\mU}{\mathcal{U}}
\newcommand{\mX}{\mathcal{X}}
\newcommand{\mY}{\mathcal{Y}}

\newcommand{\Pn}[1]{\mathbb{P}^{\, #1}}
\newcommand{\bfa}{\mathbf{a}}
\newcommand{\bfb}{\mathbf{b}}
\newcommand{\bfc}{\mathbf{c}}
\newcommand{\bfe}{\mathbf{e}}

\newcommand{\ubfa}{\underline{\mathbf{a}}}
\newcommand{\bfs}{\mathbf{s}}
\newcommand{\bfeps}{\boldsymbol\epsilon}

\newcommand{\la}{\langle}
\newcommand{\ra}{\rangle}

\newcommand{\bfy}{\mathbf{y}}
\newcommand{\bfz}{\mathbf{z}}
\newcommand{\bx}{\mathbf{x}}
\newcommand{\bft}{\mathbf{t}}

\renewcommand{\k}{\Bbbk}
\newcommand{\kxo}{\k[x_0,\ldots,x_n]}
\newcommand{\hf}[1]{{\rm HF}_{#1}}
\newcommand{\reg}[1]{{\rm reg}(#1)}

\def\iff{if and only if}

\theoremstyle{plain}
\newtheorem{theorem}{Theorem}[section]
\newtheorem{lemma}[theorem]{Lemma}
\newtheorem{corollary}[theorem]{Corollary}
\newtheorem{proposition}[theorem]{Proposition}

\theoremstyle{definition}
\newtheorem{definition}[theorem]{Definition}
\newtheorem{example}[theorem]{Example}
\newtheorem{remark}[theorem]{Remark}

\title[Regularity of toric varieties with at most one singular point]{Castelnuovo-Mumford regularity of toric varieties with at most one singular point}

\author[I. García-Marco]{Ignacio García-Marco\,\orcidlink{0000-0003-4993-7577}}
\address{Instituto de Matem\'aticas y Aplicaciones (IMAULL), Secci\'on de Matem\'aticas, Fa\-cul\-tad de Ciencias, Universidad de La Laguna, 38200, La Laguna, Spain}
\email{iggarcia@ull.edu.es}

\author[P. Gimenez]{Philippe Gimenez\,\orcidlink{0000-0002-5436-9837}}
\address{Instituto de Investigaci\'on en Matem\'aticas de la Universidad de Valladolid (IMUVA), Universidad de Valladolid, 47011 Valladolid, Spain.}
\email{pgimenez@uva.es}

\author[M. Gonz\'alez-S\'anchez]{Mario Gonz\'alez-S\'anchez\,\orcidlink{0000-0002-6458-7547}}
\address{Instituto de Investigaci\'on en Matem\'aticas de la Universidad de Valladolid (IMUVA), Universidad de Valladolid, 47011 Valladolid, Spain.}
\email{mario.gonzalez.sanchez@uva.es}
 
\thanks{
This work was supported in part by the grant PID2022-137283NB-C22 funded by MICIU/AEI/ 10.13039/501100011033 and by ERDF/EU.
The third author thanks financial support from European Social Fund, {\it Programa Operativo de Castilla y Le\'on}, and {\it Consejer\'ia de Educaci\'on de la Junta de Castilla y Le\'on}.  
Financial support of the Department of Education of the Junta de Castilla y León and FEDER Funds is gratefully acknowledged (Reference: CLU-2025-1-02- IMUVA).
Part of this work was done during a research visit of the third author to Universidad de La Laguna funded by the MICIU grant RED2022-134947-T\!
}



\begin{document}

\begin{abstract}
We establish upper bounds for the Castelnuovo--Mumford regularity of the coordinate ring of a simplicial projective toric variety with at most one singular point. In the smooth case, our results recover the bound of Herzog and Hibi [Proc. Amer. Math. Soc. 131 (2003), 2641--2647], and therefore the Eisenbud--Goto bound. Furthermore, when the variety has exactly one singular point and dimension at least $3$, we prove that its regularity also satisfies the Eisenbud--Goto bound. The proof combines combinatorial and homological methods: we study the asymptotic behavior of the sumsets associated to the toric variety and relate it to Castelnuovo--Mumford regularity via a Hochster-like formula.
\end{abstract}

\maketitle

\section*{Introduction}

Let $\mA = \{\bfa_0,\bfa_1,\ldots,\bfa_n\}$ be a finite subset of $\N^d$ for some $d\geq 1$ and $n\geq 0$, where $\bfa_i = (a_{i1},\ldots,a_{id})$ for all $i\in \{0,\ldots,n\}$. 
Given $s\in \N$, the \textit{$s$-fold sumset} of $\mA$, $s\mA$, is defined by 
$0\mA := \{\mathbf{0}\}$ and, for $s\geq 1$, \[s\mA := \{\bfa_{i_1}+\dots+\bfa_{i_s}: 0\leq i_1\leq \dots\leq i_s \leq n\}.\] 
Additive combinatorics studies the sumsets of $\mA$ and their cardinality. Khovanskii proved in \cite{Khovanskii1992} that the function $\N \rightarrow \N$ defined by $s\mapsto |s\mA|$ is asymptotically polynomial.
Later, Colarte-Gómez, Elias and Miró-Roig in \cite{Colarte2022} and, independently, Eliahou and Mazumdar in \cite{EM2022}, gave new proofs of Khovanskii's result. In \cite{Colarte2022}, the authors associate to $\mA$ a projective toric variety whose Hilbert function coincides with the function $s\mapsto |s\mA|$ as follows. 
Set $D := \max \{ |\bfa_i| = \sum_{j=1}^d a_{ij} : i=0,\ldots,n\}$, and define $\uA = \{\ubfa_0,\ldots,\ubfa_n\} \subset \N^{d+1}$ by setting $\ubfa_i = (D-|\bfa_i|,a_{i1},\ldots,a_{id}) \in \N^{d+1}$ for all $i$. 
Given an algebraically closed field $\k$, the projective toric variety $\mX_\uA \subset \P_\k^n$ determined by $\uA$ is $\mX_\uA = V(I_\uA) \subset \Pn{n}_\k$, where $I_\uA \subset R:=\kxo$ is the toric ideal determined by $\uA$, i.e., the kernel of the homomorphism of $\k$-algebras 
\[\begin{array}{rcl}
R & \rightarrow & \k[t_0,\ldots,t_d] \\
  x_i   & \mapsto & \bft^{\ubfa_i} = t_0^{a_{i0}} t_1^{a_{i1}}\dots t_d^{a_{id}}.
\end{array}\]
The homogeneous coordinate ring of $\mX_\uA$ is $\k[\mX_\uA] = \kxo / I_\uA$. In \cite[Prop.~3.3]{Colarte2022}, the authors showed that the Hilbert function of $\k[\mX_\uA]$ satisfies $\hf{\k[\mX_\uA]}(s) = |s\mA|$ for all $s\in \N$, establishing a bridge between additive combinatorics and commutative algebra. \newline

In the case $d=1$, the relation between the combinatorics of the sumsets of $\mA$ and algebraic properties of the projective monomial curve $\mC = \mX_\uA$ is explored in \cite{Elias2022,GG2023}. The goal of this paper is to extend some of these results when $d\geq2$ in some specific cases. On the one hand, we are interested here in the structure of the sumsets of $\mA$; on the other hand, in the Castelnuovo-Mumford regularity of $\k[\mX_\uA]$ and its relation with the sumsets of $\mA$. 
\newline

Suppose that $d\geq 2$ and 
let $\mS_\uA \subset \N^{d+1}$
be the affine semigroup generated by $\uA$, i.e., $\mS_\uA = \{\lambda_0 \ubfa_0+\dots+\lambda_n \ubfa_n \mid \lambda_0,\ldots,\lambda_n \in \N\}$. The semigroup $\mS_\uA$ is said to be \textit{simplicial} when the rational cone spanned by $\uA$ has dimension $d+1$ and is minimally generated by $d+1$ rays, and the toric variety $\mX_\uA$ is called \textit{simplicial} when $\mS_\uA$ is simplicial; we restrict ourselves to this case.
Then, one can assume without loss of generality that the cone spanned by $\mS_\uA$ is the non-negative orthant, i.e., $\{D\bfeps_0,\ldots,D\bfeps_d\} \subset \uA$, where $\{\bfeps_0,\ldots,\bfeps_d\}$ is the canonical basis of $\N^{d+1}$. In terms of the set $\mA$, this means that $\{\mathbf{0},D\bfeps_1',\ldots,D\bfeps_d'\} \subset \mA$, where $\{\bfeps_1',\ldots,\bfeps_d'\}$ is the canonical basis of $\N^d$. Reordering the elements of $\mA$ if necessary, we will assume that $\mA = \{\bfa_0,\ldots,\bfa_n\}$ with $\bfa_0 = \mathbf{0}$, $\bfa_i = D\bfeps_i'$ for all $i\in \{1,\ldots,d\}$, and $|\bfa_j| \leq D$ for all $j \in \{d+1,\ldots,n\}$. \newline

Under the hypotheses described in the previous paragraph, if the group generated by $\mA-\mA$ is $\Z^d$, Curran and Goldmakher describe in \cite[Thm.~1.3]{Curran2021} the sumsets of $\mA$ asymptotically. In particular, they show that for any integer $s\geq (d+1)D^d-2-2d$, 
\begin{equation}\label{eq:sumsetreg} 
s\mA = \la \mA\ra \cap \left( \bigcap_{i=1}^d \left( sD\bfeps_i' + T_i (\mA) \right) \right),
\end{equation}
where $T_i(\mA) \subset \Z^d$ is the subsemigroup of $\Z^d$ generated by $\mA-D\bfeps_i'$ for all $i\in \{1,\ldots,d\}$.
\newline

We study the sumsets of $\mA$ when the corresponding simplicial projective toric variety $\mX_\uA \subset \P_\k^n$ has at most one singular point (i.e., it is either smooth or has a single singular point). We begin in Section \ref{sec1} translating these conditions into the shape of the set $\mA$; see Theorem \ref{thm:simplicial_smooth} for the smooth case, and Proposition \ref{prop:charact_1singularpt} for the case of a single singular point. 
Then, in Section \ref{sec2}, we define the notion of sumsets regularity of $\mA$, $\sigma(\mA)$, as the smallest integer $s \in \N$ such that for all $s' \geq s$ the sumsets $s'\mA$ behave in a predictable way (see Definitions \ref{def:sumsets_reg_smooth} and \ref{def:sumsets_reg_1psing} for the precise definition in both the smooth and the case of a single singular point); in particular, for all $s \geq \sigma(\mA)$, the equality described in (\ref{eq:sumsetreg}) holds. In the main results of this section we provide upper bounds for $\sigma(\mA)$; indeed, we prove that $\sigma(\mA) \leq d(D-2)$ in the smooth case (Theorem \ref{thm:sigma_smooth}), and $\sigma(\mA) \leq d (D-2) D$ in the one single singular point case (Theorem \ref{thm:sigma_1psing_general}).
\newline

We study the Castelnuovo-Mumford regularity of $\k[\mX_\uA]$ when the  simplicial projective toric variety $\mX_\uA$ has at most one singular point. Given a homogeneous ideal $I \subset R$, one of the most relevant algebraic invariants of $R/I$ is its  Castelnuovo-Mumford regularity, ${\rm reg}(R/I)$ (see \cite{BM93} or \cite{EisenbudGoto1984} for several equivalent definitions). This invariant is a complexity measure of $R/I$ and provides an upper bound for the degrees of the syzygies in a minimal graded free resolution of $R/I$ as $R$-module. In 1984, Eisenbud and Goto \cite{EisenbudGoto1984} predicted that ${\rm reg}(R/I)$ should be bounded above by the multiplicity of $R/I$  for
non-degenerate prime ideals over an algebraically closed field, more precisely they conjectured the following:

\begin{quote} Suppose that $\k$ is an algebraically closed field. Let $I \subset R$ be a homogeneous prime ideal such that $I \subset \la x_0,\ldots,x_n \ra^2$. Then,
\[\reg{R/I} \leq \deg(R/I) - {\rm codim}(I) \, ,\]
where $\deg(R/I)$ is the degree/multiplicity of $R/I$, and ${\rm codim}(I) = n+1-\dim(R/I)$ is the codimension/height of $I$.
\end{quote}

This conjecture was refuted by Peeva and McCullough in 2018 \cite{McCullough2018} by providing examples exhibiting that the regularity of non-degenerate homogeneous prime ideals cannot be bounded above by any polynomial function
of the multiplicity of $R/I$. Later, Han and Kwak \cite{HanKwak} constructed surface counterexamples to the Eisenbud–Goto conjecture in $\mathbb P^4$.
These results motivate the study of reasonable upper bounds for the Castelnuovo-Mumford regularity within restricted families of homogeneous prime ideals. For example, it would be interesting to understand under which additional hypotheses the Eisenbud-Goto bound holds. When for a projective variety $\mX \subset \P_\k^n$ we have that its (prime, homogeneous) vanishing ideal $I = I(\mX) \subset R$ satisfies that $\reg{R/I} \leq \deg(R/I) - {\rm codim}(I)$, we will say that \textit{$\mX$ satisfies the Eisenbud-Goto bound}. Gruson, Lazarsfeld and Peskine proved in 1983 \cite{GLP} that reduced and irreducible projective curves satisfy the Eisenbud-Goto bound (see also \cite{N2014} and \cite{GG2023} for elementary combinatorial proofs of this result for monomial curves). 
The bound also holds for arithmetically Cohen–Macaulay varieties \cite{EisenbudGoto1984}, and hence for projectively normal toric varieties.
In the toric setting, Herzog and Hibi \cite{Herzog2003} proved an upper bound for the Castelnuovo--Mumford regularity of smooth simplicial projective toric varieties, implying in particular that they satisfy the Eisenbud--Goto bound.
\newline

In Section \ref{sec3}, we begin by relating $\sigma(\mA)$, the sumsets regularity of $\mA$, to ${\rm reg}(\k[\mX_\uA])$, the Castelnuovo-Mumford regularity of the coordinate ring of the corresponding toric variety, whenever $\mX_\uA$ has at most one singular point. 
For this purpose, our main tool is the Hochster-like formula for the Betti numbers (of the coordinate ring) of a toric variety studied in \cite{Briales2000} which, as a byproduct,  describes the Castelnuovo-Mumford regularity of the coordinate ring of a projective toric variety in terms of the reduced homology of a family of simplicial complexes associated to the semigroup $\mS_\uA$. Using this, we prove in Theorem \ref{thm:reg_sigma_smooth} that $\reg{\k[\mX_\uA]} = \sigma(\mA)$ whenever $\mX_\uA$ is smooth, and in Theorem \ref{thm:reg_sigma_1psing}  that  $\reg{\k[\mX_\uA]} \leq \sigma(\mA) + 1$ whenever $\mX_\uA$ has a single singular point.
As a direct consequence, the results in Section \ref{sec2} yield upper bounds on ${\rm reg}(\k[\mX_\uA])$; see Corollaries \ref{cor:reg_smooth} and \ref{cor:reg_1sing}, respectively. In particular, we provide an alternative proof of the results of \cite{Herzog2003} for $\mX_A$ smooth. Also, we obtain the following novel result: simplicial projective toric varieties of dimension $d \geq 3$ with at most one singular point satisfy the Eisenbud-Goto bound. \newline

Throughout the paper, we illustrate our results with various examples. Since, as noted above, the interesting case is the non Cohen-Macaulay one, all our examples are non Cohen-Macaulay, with the exception of Example~\ref{ej:sumsetreg_Veronese}.

\section{Simplicial projective toric varieties with one singular point}\label{sec1}

We assume that  $\k$ is an algebraically closed field and consider $\mX \subset \P_\k^{n}$ a simplicial projective toric variety with at most one singular point. Let $\mX = \mX_\mB = V(I_\mB)$, where $\mB = \{\bfb_0,\ldots,\bfb_n\} \subset \N^{d+1}$ satisfies that $\bfb_i = (b_{i0},\ldots,b_{id})$, $|\bfb_i| = \sum_{j = 0}^d b_{ij} = D$ for all $i \in \{0,\ldots,n\}$ and $\bfb_i = D \bfeps_i$ for $i \in \{0,\ldots,d\}$, with $\{\bfeps_0,\ldots,\bfeps_d\}$ the canonical basis of $\N^{d+1}$. In this section, we reformulate the condition that $\mX$ has at most one singular point in terms of the set $\mB$.

The affine charts of $\mX_\mB$ are $\{\mX_\mB\cap \mU_i\}_{i=0}^n$, where $\mU_i = \Pn{n}_\k \setminus V(x_i)$ for all $i=0,\ldots,n$. Moreover, since $\mX_\mB$ is simplicial, the variety is already covered by the first $d+1$ charts:
\begin{equation}
\mX_\mB = \cup_{i=0}^d \left( \mX_\mB \cap \mU_i \right) \, .
\label{eq:proy_aff_charts}
\end{equation}
Hence, the study of the singularities of $\mX_\mB$ reduces to these affine charts.
Indeed, suppose that $P = (p_0:\dots:p_n) \in \mX_\mB$ and $p_i \notin \mU_i$ for all $i=0,\ldots,d$. Then, $p_0=\dots=p_d=0$. For all $j=d+1,\ldots,n$,  the binomial $f_j = x_j^D - \prod_{k=0}^d x_k^{b_{jk}}$ belongs to $I_\mB$. Since $f_j(P) = 0$, then $p_j=0$ for all $j=d+1,\ldots,n$, which is impossible. This proves \eqref{eq:proy_aff_charts}. For all $i=0,\ldots,d$ and all $j=d+1,\ldots,n$, denote 
\[\bfb_j^{(i)}  \coloneq (b_{j,1},\ldots,b_{j,i-1},b_{j,i+1},\ldots,b_{j,d}) \in \N^d \, ,\]
and \begin{equation}\label{eq:beforelemma}
    \mB^{(i)} = \{D\bfeps_1',\ldots,D\bfeps_d',\bfb_{d+1}^{(i)},\ldots,\bfb_n^{(i)}\} \subset \N^d,
     \end{equation} where $\{\bfeps_1',\ldots,\bfeps_d'\}$ is the canonical basis of $\N^d$. With these notations, the affine chart $\mX_\mB \cap \mU_i$ is homeomorphic to
the simplicial affine toric variety $\mY_i  \coloneq V \left( I_{\mB^{(i)}} \right)$, for all $i=0,\ldots,d$. The following results characterize smoothness of affine toric varieties and simplicial projective toric varieties.

\begin{lemma}[{\cite[Thm.~5.4]{BG2015}}] \label{lem:Cox}
Let  $\mB = \{\bfb_1,\ldots,\bfb_n\} \subset \N^d$ a finite set of nonzero vectors. Consider the affine toric variety $\mY_\mB = V(I_\mB) \subset {\mathbb A}_\k^n$ determined by $\mB$. The following statements are equivalent:
\begin{enumerate}[(a)]
    \item\label{lem:Cox_1} $\mY_\mB$ is smooth.
    \item\label{lem:Cox_2} $\mathbf{0} = (0,\ldots,0) \in {\mathbb A}_\k^n$ is a regular point of $\mY_\mB$.
    \item\label{lem:Cox_3} The affine semigroup $\mS_\mB = \la \mB \ra$ admits a system of generators with $\dim({\mathbb Q} \mB)$ elements.
\end{enumerate}
\end{lemma}

\begin{theorem} [{\cite[Thm.~2.1]{Herzog2003}}] \label{thm:simplicial_smooth}
Let $\mX \subset \Pn{n}_\k$ be a simplicial projective toric variety of dimension $d$. 
Then, $\mX$ is smooth \iff{} there exist a number $D \in \Z_{>0}$ and a set $\mB = \{\bfb_0,\ldots,\bfb_n\} \subset \N^{d+1}$, such that $|\bfb_i| = D$ for all $i=0,\ldots,n$,
\[ \{ \bfeps_i + (D-1) \bfeps_j \mid 0 \leq i , j \leq d  \} \subset \mB \, ,\]
and $\mX = \mX_\mB$.
\end{theorem}

Combining Lemma~\ref{lem:Cox} and the discussion above, one can characterize the simplicial projective toric varieties with a single singular point. We first present an example:

\begin{example} Consider the set $\mB = \{\bfb_0,\ldots,\bfb_7\} \subset \N^{3}$ with $\bfb_0 = (6,0,0), \bfb_1 = (0,6,0), \bfb_2 = (0,0,6), \bfb_3 = (0,1,5), \bfb_4 = (0,5,1), \bfb_5 = (2,0,4), \bfb_6 = (2,4,0)$ and $\bfb_7 = (4,1,1)$; let us prove that $\mX_\mB$ has a single singular point. We have that \[ \mB^{(0)} = \{(6,0),(0,6),(1,5),(5,1),(0,4),(4,0),(1,1)\}. \] Then,  $\mB^{(0)} - \{(1,5),(5,1)\}$ is the unique minimal set of generators of the affine semigroup $\la \mB^{(0)} \ra$. By Lemma \ref{lem:Cox}, the affine chart $X_\mB \cap \mU_0$ is not smooth and $P_0 = (1:0:\cdots:0)$ is a singular point of $X_\mB$.
Also, we have that \[ \mB^{(1)} = \mB^{(2)} = \{(6,0),(0,6),(0,5),(0,1),(2,4),(2,0),(4,1)\}. \] Then, $\{(2,0),(0,1)\}$ is the unique minimal set of generators of the affine semigroup $\la \mB^{(1)} \ra = \la \mB^{(2)} \ra$. By Lemma \ref{lem:Cox}, the affine charts $X_\mB \cap \mU_1$ and $\mX_\mB \cap \mU_2$ are both smooth. Hence, we conclude that $P_0$ is the only singular point of $\mX_\mB$.

\end{example}

\begin{proposition} \label{prop:charact_1singularpt}
Let $\mB = \{\bfb_0,\ldots,\bfb_n\} \subset \N^{d+1}$ a set of nonzero vectors with $\bfb_i = (b_{i0},\ldots,b_{id})$ for all $i$. 
Assume that $\bfb_i = D\bfeps_i$, for $i = 0,\ldots,d,$ and $|\bfb_i| = D$ for all $i=0,\ldots,n$, for some $D \in \Z_{>0}$, and denote by $\mX_\mB = V(I_\mB) \subset \Pn{n}_\k$ the simplicial projective toric variety determined by $\mB$.
\begin{enumerate}[(1)]
\item\label{prop:charact_1singularpt_1} If $\mX$ has exactly one singular point, that point is $P_i = (0:\dots:0:1^{(i)}:0:\dots:0)$, for some $i \in \{0,\ldots,d\}$.

\item\label{prop:charact_1singularpt_2} The only singular point of $\mX$ is $P_0$ if and only if 
\[\{(D-1) \bfeps_i + \bfeps_j \, \vert \, 1 \leq i, j \leq d\} \cup \{e \bfeps_0 + (D-e) \bfeps_j \, \vert \, 1 \leq j \leq d\} \subset \mB \, ,\]
where $e\in \Z_{>0}$ is a divisor of $D$ that divides $b_{i0}$ for all $i\in \{0,\ldots,n\}$, and if $e=1$ then there exists $j \in \{1,\ldots,d\}$ such that  $(D-1) \bfeps_0 + \bfeps_j \notin \mB$.
\end{enumerate}
\end{proposition}

\begin{figure}[htbp]
\centering
\includegraphics[width=0.7\textwidth]{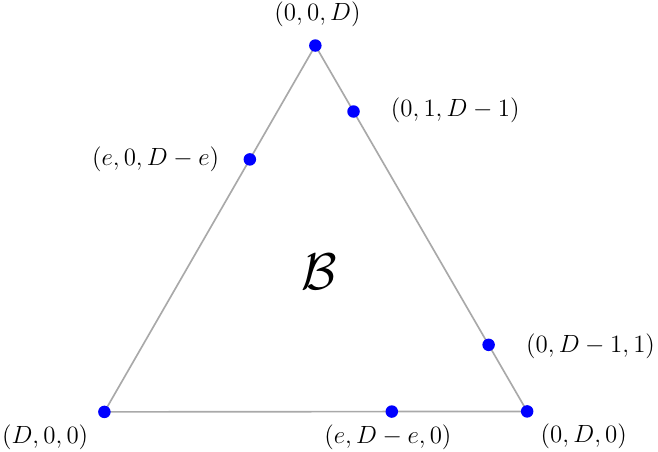}  
\caption{Shape of a set $\mB$ in Proposition~\ref{prop:charact_1singularpt} if $e\neq D$ and $d = 2$.}
\label{fig:1ptosing}
\end{figure}

\begin{remark}\label{rem:sup1psing}
In part \ref{prop:charact_1singularpt_2} of Proposition~\ref{prop:charact_1singularpt}, we distinguish two different behaviors.
\begin{enumerate}[(i)]
\item\label{rem:sup1psing_i} \underline{$e<D$}: In this case, $e \bfeps_0 + (D-e) \bfeps_j \neq e \bfeps_0 + (D-e) \bfeps_i $ for $i \neq j$.
\item\label{rem:sup1psing_ii} \underline{$e=D$}: In this case, $e \bfeps_0 + (D-e) \bfeps_j = D \bfeps_0$  for $j = 1,\ldots,d$, and for all $\bfb_i \in \mB$, such that $\bfb_i \neq (D,0,\ldots,0)$, one has that $b_{i0} = 0$. Hence, $\mB = \{D\bfeps_0\} \cup \left( \{0\} \times \mB' \right)$, where 
$I_{\mB'} \subset \k[x_1,\ldots,x_n]$ is the defining ideal of a smooth simplicial projective toric variety, by Theorem~\ref{thm:simplicial_smooth}, and $I_\mB = I_{\mB'} . \k[x_0,\ldots,x_n]$ is the extension of $I_{\mB'}$.
Therefore, the resolutions of $\k[x_0,\ldots,x_n]/I_\mB$ and $\k[x_1,\ldots,x_n]/I_{\mB'}$ are identical, and hence these two modules have the same Castelnuovo-Mumford regularity.  Geometrically, $\mX_\mB$ is a cone over the smooth variety $\mX_{\mB'}$.
\end{enumerate}
\end{remark}

\begin{proof}[Proof of Prop.~\ref{prop:charact_1singularpt}]
For $i \in \{0,\ldots,d\}$, let $\mY_i$ be the $i$-th affine chart of $\mX$, i.e. $\mY_i = V\left( I_{\mB^{(i)}} \right)$, where $\mB^{(i)} \subset \N^d$ is the set defined in (\ref{eq:beforelemma}). 
If there are two non-smooth affine charts, then, by Lemma~\ref{lem:Cox}~\ref{lem:Cox_2}, $\mX$ has at least two singular points. Thus, there is exactly one singular affine chart.
Again, by Lemma~\ref{lem:Cox}~\ref{lem:Cox_2}, the singular affine chart is $\mY_k$ \iff{} the only singular point is $P_k$. This proves \ref{prop:charact_1singularpt_1}. 

Assume now that $P_0$ is the only singular point of $\mX$. Moreover, assume that $\gcd\left( \{b_{\ell i} \mid  0\leq \ell \leq n, 0\leq i \leq d\} \right)=1$.
For all $0\leq i,j\leq d$, $i\neq j$, let 
\[\lambda_{ij}  \coloneq \min \{ \lambda \in \Z_{>0} \mid (D-\lambda) \bfeps_i + \lambda \bfeps_j \in \mB \}.\]
Since $\mY_1$ is smooth, by Lemma~\ref{lem:Cox}~\ref{lem:Cox_3} one has that \[ \{(\lambda_{10},0,\ldots,0),(0,\lambda_{12},0,\ldots,0), \ldots,(0,\ldots,0,\lambda_{1d})\} \] is the minimal generating set of $\langle \mB^{(1)} \rangle$. Hence, $\lambda_{1j} \mid b_{\ell j}$ for  $j \in \{0,2,\ldots,d\}$, $\ell \in \{0,\ldots,n\}$. As a consequence,
\begin{enumerate}[(a)]
\item\label{proof:prop_1psing_a} $\lambda_{10} \mid \lambda_{i0}$ for  $i \in \{1,\ldots,d\}$, 
\item\label{proof:prop_1psing_b} $\lambda_{1j} \mid \lambda_{ij}$ for $i,j \in \{2,\ldots,d\}, i \neq j$, and
\item\label{proof:prop_1psing_c} $\lambda_{1i} \mid \lambda_{i1}$ for  $i \in \{2,\ldots,d\}$.
\end{enumerate}

Indeed, \ref{proof:prop_1psing_a}, and \ref{proof:prop_1psing_b} are straightforward. To prove \ref{proof:prop_1psing_c} it suffices to observe that $\lambda_{1i} \mid D$ and $\lambda_{1i} \mid D - \lambda_{i1}$, so, $\lambda_{1i} \mid \lambda_{i1}$. By symmetry for all $i \in \{1,\ldots,d\}$, from \ref{proof:prop_1psing_a}, \ref{proof:prop_1psing_b} and \ref{proof:prop_1psing_c}, we deduce that 
\begin{enumerate}[(a)]
\setcounter{enumi}{3}
\item\label{proof:prop_1psing_d} $\lambda_{10} = \lambda_{i0}$ for  $i \in \{1,\ldots,d\}$, 
\item\label{proof:prop_1psing_e} $\lambda_{ij} = \lambda_{kl}$ for  $i,j,k,l \in \{1,\ldots,d\}, i \neq j, k \neq l$.
\end{enumerate}

By \ref{proof:prop_1psing_e} we have that $\lambda_{12}$ divides $b_{\ell j}$ for all $j \in \{1,\ldots,d\}$ and all $\ell \in \{0,\ldots,n\}$. In particular, $\lambda_{12}$ divides $D$ and, therefore, $\lambda_{12}$ divides $b_{\ell 0} = |\bfb_\ell| - \sum_{i = 1}^d b_{\ell i}$. We conclude that  $\lambda_{12} \mid \gcd\left( \{b_{\ell i} \mid  0\leq \ell \leq n, 0\leq i \leq d\} \right)=1$. Therefore, $\lambda_{ij} = \lambda_{12} = 1$, for all $i,j\in \{1,\ldots,d\}$, $i\neq j$.

Hence, setting $e := \lambda_{10}$ ($= \lambda_{i0}$ for all $i\in \{1,\ldots,d\}$), we have proved that 
\[\{(D-1) \bfeps_i + \bfeps_j \, \vert \, 1 \leq i, j \leq d\} \cup \{e \bfeps_0 + (D-e) \bfeps_j \, \vert \, 1 \leq j \leq d\} \subset \mB \, ,\]
$1\leq e \leq D$ is a divisor of $D$, and $e$ divides $b_{i0}$ for all $i$.
Finally, if $e=1$ there exists $j \in \{1,\ldots,d\}$ such that $(D-1) \bfeps_0 + \bfeps_j \notin \mB$; otherwise, by  Theorem~\ref{thm:simplicial_smooth}, $\mX$ is smooth.

Conversely, by Lemma~\ref{lem:Cox}~\ref{lem:Cox_3}, the affine charts $\mY_1,\ldots,\mY_d$ are smooth, since $\la \mB^{(i)} \ra = \la (e,0,\ldots,0),(0,1,0,\ldots,0),\ldots,(0,\ldots,0,1) \ra$ for all $i \in \{1,\ldots,d\}$. 
Thus, every point in $\mX$ different from $P_0 = (1:0:\dots:0)$ is a regular point.
Moreover, $\mY_0$ is the affine toric variety determined by $\mB^{(0)}$. By Lemma \ref{lem:Cox}, $\mY_0$ is not smooth and, hence, $P_0$ is the only singular point of $\mX$. 
\end{proof}

\section{Structure of the sumsets: sumsets regularity}\label{sec2}

Consider a finite set $\mA = \{\bfa_0,\ldots,\bfa_n\} \subset \N^d$, $n \geq d\geq 1$, with $\bfa_i = (a_{i1},\ldots,a_{id})$  and $|\bfa_i| \leq D$ for $i\in \{0,\ldots,n\}$, and assume that $\bfa_0 = \mathbf{0}$ and $\bfa_i = D\bfeps_i'$, $i\in \{1,\ldots,d\}$, for some $D \in \N$, $D\geq 2$. We define the homogenization of $\mA$, denoted by $\uA$, as the set $\uA = \{\ubfa_0,\ldots,\ubfa_n\} \subset \N^{n+1}$, where $\ubfa_i = (D-|\bfa_i|,a_{i1},\ldots,a_{id})$ for all $i\in \{0,\ldots,n\}$.
As usual, $\k$ is an algebraically closed field, and $\mX_\uA \subset \Pn{n}_\k$ is the simplicial projective toric variety determined by $\uA$. \newline

The sumsets of $\mA$ and $\uA$ are related as follows. Let  $\bfb \in  \N^d$, then:
\begin{equation}\label{eq:relationsumsets}
    \bfb = (b_1,\ldots,b_d) \in  s \mA \text{ if and only if } (Ds-|\bfb|,b_{1},\ldots,b_{d}) \in s \uA.
\end{equation}

In this section, we provide structure theorems for the sumsets of $\mA$ for specific classes of sets $\mA$. More specifically, we are interested in the sets $\mA$ such that $\mX_\uA$ is either smooth (Subsection~\ref{subsec:structure_smooth}) or has a single singular point (Subsection~\ref{subsec:structure_1psing}).

\subsection{The smooth case}\label{subsec:structure_smooth}

Let $\mA = \{\bfa_0,\ldots,\bfa_n\} \subset \N^d$ be a finite set, and suppose that $\uA$ defines a smooth simplicial projective toric variety $\mX_\uA \subset \Pn{n}_\k$. Hence, without loss of generality, we can assume that $\uA \subset \N^{d+1}$ satisfies the conditions of Theorem~\ref{thm:simplicial_smooth} or, equivalently, that
\[\{\mathbf{0},\bfeps_i',(D-1)\bfeps_i', \bfeps_i'+(D-1)\bfeps_j' \mid 1 \leq i,j\leq d\} \subset \mA \, .\]

For each $s\in \N$, denote $\Delta_s \coloneq \{(y_1,\ldots,y_d) \in \N^d \mid y_1+\dots+y_d \leq sD\}$. Note that for $s\geq 1$, $\Delta_s$ is the set of lattice points of a simplex. Moreover,  $s\mA \subset \Delta_s$ for all $s\in \N$. We will prove that the two sets eventually coincide. When $D = 2$, then $\mA = \Delta_1$ and we clearly have that $s \mA = \Delta_s$ for all $s \in \N$.

\begin{proposition} \label{prop:smooth_lleno}
Let $\mA = \{\bfa_0,\ldots,\bfa_n\} \subset \N^d$, $n\geq d \geq 1$, $D \geq 3$  be a set such that
\[\{\mathbf{0},\bfeps_i',(D-1)\bfeps_i', \bfeps_i'+(D-1)\bfeps_j' \mid 1 \leq i,j\leq d\} \subset \mA.\]
Then, $d(D-2)\mA = \Delta_{d(D-2)}$.
\end{proposition}

\begin{proof}
Consider $\mB_d := \{\mathbf{0},\bfeps_i',(D-1)\bfeps_i', \bfeps_i'+(D-1)\bfeps_j' \mid 1 \leq i,j\leq d\} \subset \N^d$ and let us prove that \[ d(D-2)\mB_d = \Delta_{d(D-2)}, \]
by induction on $d \geq 1$. As a consequence of this, we have that 
$\Delta_{d(D-2)} = d(D-2)\mB_d \subset d(D-2) \mA \subset \Delta_{d(D-2)}$, and the result follows.

 For $d=1$, we have to show that $y \in (D-2) \mB_1$ for all $y \leq D(D-2)$. We write 
$y = \lambda D-\mu$ with $1\leq \lambda \leq D-2$ and $0\leq \mu \leq D-1$ and separate two cases. If  $0\leq \mu \leq \lambda$, we write $\bfy = \mu (D-1) + (\lambda - \mu) D$ and we have that  $y \in s\mB_1$ for $s=\mu+(\lambda-\mu) = \lambda \leq (D-2)$, so $y \in s \mB_1 \subset  (D-2)\mB_1$.
If $\lambda < \mu \leq D-1$, we write $ y= (\lambda - 1) D + (D - \mu) 1$ and we have that $y \in s\mB_1$ for $s=D-\mu+\lambda -1 \leq D-2$, since $-\mu+\lambda \leq -1$, so $y \in (D-2)\mB_1$.

 Let $d\geq 2$ and take $\bfy = (y_1,\ldots,y_d) \in \N^d$, such that $\bfy \in \Delta_{d(D-2)}$ or, in other words,  $|\bfy| \leq d(D-2)D$, and let us show that $\bfy \in d(D-2)\mB_d$. 

\underline{Case 1:} $|\bfy| \leq (d-1)(D-2)D$. We take $j \in \{1,\ldots,d\}$ such that $y_j \leq (D-2)D$ and write $\bfy = \bfy' +  y_j \bfeps_j'$, where $\bfy' = (y_1,\ldots,y_{j-1},0,y_{j+1},\ldots,y_d)$. By  the case $d=1$, 
we have that $y_j \in (D-2) \mB_1$, so $y_j \bfeps_j' \in (D-2)\mB_d$. Also, if one considers the $j$-th projection map $\pi_j: \N^d \longrightarrow \N^{d-1}$ defined as $\pi_j(z_1,\ldots,z_j,\ldots,z_d) = (z_1,\ldots,z_{j-1},z_{j+1},\ldots,z_d)$ and $\mB_d' := \{\bfa = (a_1,\ldots,a_d)\in \mB_d : a_j = 0\}$ one has that $\pi_j(\mB_d') = \mB_{d-1}$ and $|\pi_j(\bfy')| \leq (d-1)(D-2)D$. By the inductive hypothesis, it follows that $\pi_j(\bfy') \in (d-1)(D-2) \mB_{d-1}$ and, thus, $\bfy' \in (d-1)(D-2) \mB_{d}$. We conclude that $\bfy = \bfy'+ y_j \bfeps_j' \in d(D-2)\mB_d$.

\underline{Case 2:} $|\bfy| > (d-1)(D-2)D$. We begin by proving the following claim.

{\it Claim:} If $|\bfz| = (d-1)(D-2)D$, then $\bfz \in (d-1)(D-2)\mB_d$ for all $d \geq 1$. 

{\it Proof of the claim:} If $d=1$, the result is clear. Assume $d\geq 2$ and consider $\pi_1(\bfz) = (z_2,\ldots,z_{d}) \in \N^{d-1}$. We have that $\mB_{d-1} =\pi_1 \left( \{\bfa \in \mB_d : |\bfa| = D\} \right) \subset \N^{d-1}$. Since $|\pi_1(\bfz)| \leq (d-1)(D-2)D$, then by the inductive hypothesis $\pi_1(\bfz) \in (d-1)(D-2) \mB_{d-1}$. Thus, one gets that 
\[\bfz = \left( (d-1)(D-2)D-|\pi_1(\bfz)|,z_2,\ldots,z_d \right) \in (d-1)(D-2)\mB_d,\]
which proves the claim.

We write $|\bfy| = (d-1)(D-2)D+\lambda D-\mu$ with $1\leq \lambda \leq D-2$ and $0\leq \mu \leq D-1$; we separate two subcases:

\underline{Case 2.1:} $0\leq \mu \leq \lambda$. Take $\bfb_1,\ldots,\bfb_\mu \in \mA$ with $|\bfb_i| = D-1$ for all $i$, and $\bfc_1,\ldots,\bfc_{\lambda-\mu} \in \mA$ with $|\bfc_j| = D$ for all $j$, such that $\bfz = \bfy-\sum_{i=1}^\mu \bfb_i - \sum_{j=1}^{\lambda-\mu} \bfc_j \in \N^d$. Then, $|\bfz| = (d-1)(D-2)D$, and hence $\bfz \in (d-1)(D-2)\mA$ by the previous claim. Therefore, $\bfy \in s\mA$ for $s=(d-1)(D-2)+\mu+(\lambda-\mu) \leq d(D-2)$, so $\bfy \in d(D-2)\mA$.

\underline{Case 2.2:} $\lambda < \mu \leq D-1$. Take $\bfb_1,\ldots,\bfb_{D-\mu} \in \mA$ with $|\bfb_i| = 1$ for all $i$, and $\bfc_1,\ldots,\bfc_{\lambda-1} \in \mA$ with $|\bfc_j| = D$ for all $j$, such that $\bfz = \bfy-\sum_{i=1}^{D-\mu} \bfb_i - \sum_{j=1}^{\lambda-1} \bfc_j \in \N^d$. Then, $|\bfz| = (d-1)(D-2)D$, and hence $\bfz \in (d-1)(D-2)\mA$ by the previous claim. Therefore, $\bfy \in s\mA$ for $s=(d-1)(D-2)+D-\mu+\lambda -1 \leq d(D-2)$, since $-\mu+\lambda \leq -1$, so $\bfy \in d(D-2)\mA$.
\end{proof}

In the proof of Corollary \ref{cor:strucutre_smooth} we will use the following lemma. We postpone its proof until the next subsection, since it corresponds to the case $e = 1$ in Lemma \ref{lem:dependedimcone}.

\begin{lemma}\label{lem:dependedim}  
$\Delta_{s} + \{\mathbf{0}, D\bfeps_1',\ldots,D\bfeps_d'\} = \Delta_{s+1}$ if and only if $s \geq d - \frac{d}{D}$.
\end{lemma}

\begin{corollary} \label{cor:strucutre_smooth}
Let $\mA \subset \N^d$ be a finite set as in Proposition~\ref{prop:smooth_lleno}, then $s\mA = \Delta_s$ 
for all $s\geq d(D-2)$. 
Equivalently, $s \uA = \{\bfy \in \N^{d+1} : |\bfy| = sD\}$ for all $s \geq d(D-2)$.
\end{corollary}

\begin{proof}
If $D = 2$ then $0 \mA = \Delta_0 = \{\mathbf{0}\}$, $\mA = \Delta_1$ and, hence, $s \mA = \Delta_s$ for all $s \in \mathbb N$.
For $D > 2$ let us  prove  that $s\mA = \Delta_s$ by induction on $s$. 
The base case is Proposition~\ref{prop:smooth_lleno}. By Lemma \ref{lem:dependedim}, for $s \geq d(D-2) \geq d - \frac{d}{D},$ we have that \[ (s+1) \mA \subset \Delta_{s+1} = \Delta_s + \{\mathbf{0}, D\bfeps_1',\ldots,D\bfeps_d'\} \subset \Delta_s + \mA = s \mA + \mA = (s+1) \mA,\]
and we conclude that $(s+1)\mA = \Delta_{s+1}$. 
The last assertion follows from the relation between the sumsets of $\mA$ and $\uA$ described in (\ref{eq:relationsumsets}).
\end{proof}

\begin{definition} \label{def:sumsets_reg_smooth}
Let $\mA = \{\bfa_0,\ldots,\bfa_n\} \subset \N^d$, $n\geq d \geq 1$ be a set such that
\[\{\mathbf{0},\bfeps_i',(D-1)\bfeps_i', \bfeps_i'+(D-1)\bfeps_j' \mid 1 \leq i,j\leq d\} \subset \mA.\]
The {\it sumsets regularity of $\mA$} is 
\[\sigma(\mA) = \min\{s\in \N \mid s\mA = \Delta_s \text{, and } \Delta_s + \{\mathbf{0}, D\bfeps_1',\ldots,D\bfeps_d'\} = \Delta_{s+1}\}.\]
When there is no confusion, we will write $\sigma = \sigma(\mA)$.
\end{definition}

From the definition one derives that $t \mA = \Delta_t$ for all $t \geq \sigma(\mA)$. When $D = 2$, we have that $s\mA = \Delta_s$ for all $s\in \N$ and, by Lemma \ref{lem:dependedim}, it follows that $\sigma(\mA) = d - \lfloor \frac{d}{2} \rfloor$. Whenever $D \geq 3$,
Lemma \ref{lem:dependedim} and Corollary \ref{cor:strucutre_smooth} provide the following lower and upper bounds for $\sigma(\mA)$.

\begin{theorem} \label{thm:sigma_smooth}
Let $\mA = \{\bfa_0,\ldots,\bfa_n\} \subset \N^d$, $n\geq d \geq 1$, $D \geq 3$ be a set such that
\[\{\mathbf{0},\bfeps_i',(D-1)\bfeps_i', \bfeps_i'+(D-1)\bfeps_j' \mid 1 \leq i,j\leq d\} \subset \mA.\]
Then, the sumsets regularity of $\mA$ satisfies that
\[ d - \frac{d}{D} \leq \sigma(\mA) \leq d(D-2).\]
\end{theorem}

The bounds on $\sigma(\mA)$ obtained in Theorem~\ref{thm:sigma_smooth} are sharp, as the following examples show.

\begin{example} \label{ex:sigma_smooth}
Let $d \geq 1$, $D \geq 3$ and consider $\mA = \{\mathbf{0},\bfeps_i',(D-1)\bfeps_i', \bfeps_i'+(D-1)\bfeps_j' \mid 1 \leq i,j\leq d\} \subset \N^d$, where $\{\bfeps_1',\ldots,\bfeps_d'\}$ is the canonical basis of $\N^d$; then, $\sigma(\mA) = d(D-2)$.

Indeed, by Theorem~\ref{thm:sigma_smooth}, $\sigma(\mA) \leq d(D-2)$. If $(D,d) = (3,1)$, then $1 - \frac{1}{3} \leq \sigma(\mA) \leq 1$, and we get $\sigma(\mA) = 1$. Assume now that $(D,d) \neq (3,1)$. We observe that  the only way of writing  $\bfy = (D-2,D-2,\ldots,D-2) \in \N^d$ as a sum of nonzero elements in $\mA$ is $\bfy = \sum_{i=1}^d (D-2)\bfeps_i'$. Then, $\bfy \notin (d(D-2)-1) \mA$  but $\bfy \in \Delta_{d(D-2)-1}$ (because $(D,d) \neq (3,1)$); thus $\sigma(\mA) \geq d(D-2)$.

\end{example}

\begin{example} \label{ej:sumsetreg_Veronese}
Take $\mA = \{\bx \in \N^d : |\bx| \leq D\} = \Delta_1$. Then, the sumsets regularity of $\mA$ is $\sigma(\mA) = d-\lfloor \frac{d}{D} \rfloor$. Let us prove it.

One has that $s\mA = \Delta_s$ for all $s\in \N$. Also, by Lemma \ref{lem:dependedim}, $\Delta_{s} + \{\mathbf{0}, D\bfeps_1',\ldots,D\bfeps_d'\} = \Delta_{s+1}$ if and only if $s \geq d - \frac{d}{D}$. Therefore, $\sigma(\mA) =  d-\lfloor \frac{d}{D} \rfloor$.

\end{example}

The following result shows that the behavior of Corollary \ref{cor:strucutre_smooth} (i.e., $s\mA = \Delta_s$ for all $s\gg 0$) characterizes the smoothness of a simplicial projective toric variety.

\begin{theorem}\label{thm:strThmCharactSmooth}
Let $n\geq d \geq 1, D \geq 2$ and the set $\mA = \{\bfa_0,\ldots,\bfa_n\} \subset \N^d$ such that $\{\mathbf{0},\allowbreak D\bfeps_1,\ldots,\allowbreak D\bfeps_d\} \subset \mA$ and $|\bfa_i| \leq D$ for all $i \in \{0,\ldots,n\}$. 
The following conditions are equivalent:
\begin{enumerate}[(a)]
\item\label{thm:strThmCharactSmooth_1} 
$\mX_\uA \subset \Pn{n}_\k$ is smooth.
\item\label{thm:strThmCharactSmooth_2} 
$\{\bfeps_i+(D-1)\bfeps_j \mid 0 \leq i,j\leq d\} \subset \uA$, where $\{\bfeps_0,\ldots,\bfeps_d\}$ is the canonical basis of $\N^{d+1}$. 
\item\label{thm:strThmCharactSmooth_c} 
$\{\mathbf{0},\bfeps_i',(D-1)\bfeps_i', \bfeps_i'+(D-1)\bfeps_j' \mid 1 \leq i,j\leq d\} \subset \mA$, where $\{\bfeps_1',\ldots,\bfeps_d'\}$ is the canonical basis of $\N^d$.
\item\label{thm:strThmCharactSmooth_3} 
There exists $s_0 \in \N$ such that $s\mA = \Delta_s$ for all $s \geq s_0$.
\end{enumerate}
\end{theorem}

\begin{proof}
The equivalence \ref{thm:strThmCharactSmooth_1}~$\Leftrightarrow$~\ref{thm:strThmCharactSmooth_2} is Proposition~\ref{prop:charact_1singularpt}, \ref{thm:strThmCharactSmooth_2}~$\Leftrightarrow$~\ref{thm:strThmCharactSmooth_c} is direct from the construction of $\uA$, and the implication \ref{thm:strThmCharactSmooth_c}~$\Rightarrow$~\ref{thm:strThmCharactSmooth_3} is Corollary~\ref{cor:strucutre_smooth}. 
Let us prove \ref{thm:strThmCharactSmooth_3}~$\Rightarrow$~\ref{thm:strThmCharactSmooth_c}. Take $s\in \N$ such that $s\mA = \Delta_s$, and fix $i,j$ with $1\leq i,j \leq d$. Since $\bfeps_i' \in s\mA = \Delta_s$, then $\bfeps_i' \in \mA$. Moreover, since $(Ds-1)\bfeps_i'+\bfeps_j' \in s\mA = \Delta_s$, then $(D-1)\bfeps_i' + \bfeps_j' \in \mA$.
Therefore, $\mA$ is as in~\ref{thm:strThmCharactSmooth_c}. 
\end{proof}

\subsection{Varieties with one singular point}\label{subsec:structure_1psing}

Let $\mA = \{\bfa_0,\ldots,\bfa_n\} \subset \N^d$ be a finite set, and suppose that its homogenization $\uA = \{\ubfa_0,\ldots,\ubfa_n\} \subset \N^{d+1}$ defines a simplicial projective toric variety with a single singular point. Proposition~\ref{prop:charact_1singularpt} characterizes such sets $\uA$. Hence, throughout this subsection, we will assume that
\begin{equation}
\{\mathbf{0}\} \cup \{(D-e) \bfeps_i' \mid 1\leq i \leq d \} \cup \{ (D-1)\bfeps_i'+\bfeps_j' \mid 1\leq i,j \leq d \}  \subset \mA,
\label{eq:A_sup1singp}
\end{equation}
where $D \geq 2$, $|\bfa_i| \leq D$,
$1\leq e \leq D$ is a divisor of $D$ that divides $|\bfa_i|$ for all $i\in \{0,\ldots,n\}$; and if $e=1$, then there exists $j\in \{1,\ldots,d\}$ such that $\bfeps_j' \notin \mA$. In particular, this implies that whenever $D = 2$, then $e = 2$.

If we denote $\N_e^d \coloneq  \{\bfy \in \N^d : e \text{ divides } |\bfy|\}$, then  we have that $\langle \mA \rangle \subset \N_e^d$.
We will prove that $\mH :=  \N_e^d \setminus \langle \mA \rangle$  is a finite set.  Also, if one denotes $\Delta_{s,e}= \{\bfy \in \N_e^d : |\bfy| \leq sD\}$, one clearly has that $s \mA \subset \Delta_{s,e} \setminus \mH$, and we will show that these two sets eventually coincide.

If $D=2$, then $e=2$ and $\mA = \{\mathbf{0}\} \cup \{\bfeps_i'+\bfeps_j' \mid 1\leq i ,j \leq d\}$. Hence, it is clear that $s\mA = \Delta_{s,2}$ for all $s\in \N$, and we have that $\mH = \emptyset$ in this case.

\begin{proposition} \label{prop:structure_sumsets}
Let $D \geq 3$, $t_0 := (D-2)(d-1) + \frac{D}{e} - 2$ and $s_0 :=  \frac{D}{e} t_0 $, then $\mH \subset \Delta_{t_0,e}$  and $s\mA = \Delta_{s,e} \setminus \mH$ for all $s \geq s_0$. 
\end{proposition}
\begin{proof} We divide the proof in two parts:
\begin{itemize}
\item[(i)] If $\bfy \in \N^d_e \setminus \Delta_{t_0,e}$, then $\bfy \in s\mA$ for all $s \geq \frac{|\bfy|}{D}$.
\item[(ii)] If $\bfy \in \Delta_{t_0,e} \setminus \mH$, then $\bfy \in s_0 \mA.$
\end{itemize}
In particular, (i) implies that $\mH \subset \Delta_{t_0,e}$; and (i) and (ii) together imply that $s\mA = \Delta_{s,e} \setminus \mH$ for all $s \geq s_0$. 

To prove (i) we take  $\bfy \in \N^d_e \setminus \Delta_{t_0,e}$, then $|\bfy| \geq D t_0 + e$. Let $\lambda \in \N$, $\lambda \leq D/e-1$ such that $|\bfy| \equiv -\lambda e \pmod{D}$. 

\smallskip
\noindent {\it Claim 1.} $|\bfy| - \lambda (D-e) \geq (d-1)(D-2)D$

\noindent {\it Proof of claim 1.} If $\lambda = 0$, then $|\bfy| \equiv 0 \pmod{D}$ and $|\bfy| \geq Dt_0 + e$. Therefore $|\bfy| \geq Dt_0 + D \geq (d-1) (D-2) D$. If $\lambda\geq 1$, then $|\bfy|-\lambda(D-e) \geq  D t_0 + e - (\frac{D}{e} - 1) (D-e) = (d-1)(D-2)D$. 

\smallskip

\noindent {\it Claim 2.} There exist $\lambda_1,\lambda_2,\ldots,\lambda_d \in \N$ such that $\sum_{i=1}^d \lambda_i =\lambda$, and $y_i \geq \lambda_i(D-e)$ for all $i=1,\ldots,d$.

\noindent {\it Proof of claim 2.} If $\lambda = 0$ there is nothing to prove. Assume that $\lambda \geq 1$. Whenever $z= (z_1,\ldots,z_d) \in \N^d$ satisfies that $|\bfz| > d(D-e-1)$, then $z_i \geq D-e$ for some $i \in \{1,\ldots,d\}$. For all $\mu < \lambda$ we have that $|\bfy| - \mu (D-e) \geq |\bfy| - (\lambda - 1) (D-e)$ and, by {\it Claim 1}, it follows that $|\bfy| - \mu (D-e) \geq (d-1)(D-2)D + D-e > d(D-e-1)$, and {\it Claim 2} follows.

\smallskip

Consider $\bfy' := \bfy - \sum_{i=1}^d \lambda_i (D-e) \bfeps_i'$; then 
\begin{itemize} 
\item $\bfy \in \N_e^d$ by {\it Claim 2},
\item  $|\bfy'| \equiv 0 \pmod{D}$ because $|\bfy'| = |\bfy| -\lambda(D-e)$, and
\item $|\bfy'|  \geq (d-1)(D-2)D$ by {\it Claim 1}.
\end{itemize}

Take now  $\mA' = \{(D-1)\bfeps_i' + \bfeps_j' \, \vert \, 1\leq i,j \leq d\} \subset \mA$. By  Corollary~\ref{cor:strucutre_smooth}, we get that $\bfy \in \la \mA' \ra$.
Therefore, $\bfy = \sum_{1\leq i, j \leq d} \mu_{i,j} ((D-1)\bfeps_i' + \bfeps_j') + \sum_{i=1}^d \lambda_i (D-e) \bfeps_i'$ for some $\mu_{i,j}, \lambda_i \in \N$. Then, $\bfy \in \left( \sum_{i,j} \mu_{i,j}  + \lambda \right) \mA$.
Take $s\in \N$ such that $s \geq \frac{|\bfy|}{D}$ and let us prove that $\sum_{i,j} \mu_{i,j} +\lambda \leq s$. Since $sD \geq |\bfy| = \left( \sum_{i,j} \mu_{i,j} \right) D + \lambda(D-e) = \left( \sum_{i,j} \mu_{i,j}+\lambda \right) D -\lambda e$, then \[\sum_{i,j} \mu_{i,j} + \lambda \leq s+\lambda \frac{e}{D} \leq s + \left( \frac{D}{e}-1 \right) \frac{e}{D} < s+1 \, .\] Thus, $\sum_{i,j} \mu_{i,j} + \lambda \leq s$, and  (i) is proved.

Let us prove (ii). Take $\bfy \in \Delta_{t_0,e} \setminus \mH$; then $|\bfy| \leq t_0 D$ and we can write $\bfy = \sum_{i=0}^n \mu_i \bfa_i$ for some $\mu_i \in \N,\, \bfa_i \in \mA \setminus \{\bf0\}$. 
Since $|\bfy| = \sum_i \mu_i |\bfa_i| \geq \left( \sum_i \mu_i \right) e$, then 
\[\sum_i \mu_i \leq \frac{|\bfy|}{e}  \leq \frac{D}{e} t_0 = s_0\, ,\] 
so $\bfy \in (\sum_i \mu_i) \mA \subset s_0 \mA$; and the result follows.
\end{proof}

\begin{example}\label{ex:1sing} Consider the set \[ \mA = \{(0,0), (4,0), (0,4), (3,1), (1,3), (2,0), (0,2)\} \subset \N^2; \]
this set satisfies the conditions in \eqref{eq:A_sup1singp} with $d = 2$, $D = 4$ and $e = 2$. Proposition \ref{prop:structure_sumsets} ensures that $\mH \subset \Delta_{2,2}$ and that $s \mA = \Delta_{s,2} \setminus \mH$ for all $s \geq 4$.

Nevertheless, we observe that $0 \mA = \{(0,0)\}$, $\mA = \Delta_{1,2} \setminus \{(1,1), (2,2)\}$, and
$s \mA = \Delta_{s,2} \setminus \{(1,1)\}$ for all $s \geq 2$ (see Figure \ref{fig:ex1sing}). Therefore, $\mH = \{(1,1)\} \subset \Delta_{1,2}$ and $s \mA = \Delta_{s,2} \setminus \mH$ for all $s \geq 2$.

\begin{figure}
\includegraphics[scale = 1]{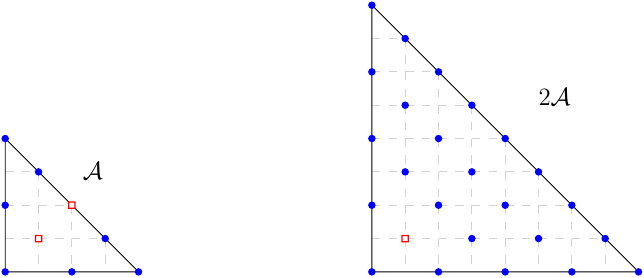}
\caption{For $\mA$ as in Example \ref{ex:1sing}, filled circles represent elements in $\mA$ and $2 \mA$, respectively; while empty squares correspond to elements in $\Delta_{1,2} \setminus \mA$ and $\Delta_{2,2} \setminus 2\mA$, respectively.} \label{fig:ex1sing}
    \end{figure}

\end{example}

\begin{definition} \label{def:sumsets_reg_1psing}
Let $\mA \subset \N_e^d$ be a set as in \eqref{eq:A_sup1singp}.
The {\it sumsets regularity of $\mA$}, $\sigma(\mA)$, is 
\[\sigma(\mA) = \min\{ s\in \N \mid \mH \subset \Delta_{s,e}\, , s'\mA = \Delta_{s',e}\setminus \mH \, \, \forall s' \geq s \, , \Delta_{s,e} + \{\mathbf{0},D\bfeps_1',\ldots,D\bfeps_d'\} = \Delta_{s+1,e} \}.\]
When there is no confusion, we will write $\sigma = \sigma(\mA)$.
\end{definition}

\begin{remark}\label{rem:defsigma}
\begin{enumerate}[(1)]
\item\label{rem:defsigma_1} 
If one allows $e=1$ and $\{\bfeps_i' \mid 1\leq i \leq d\} \subset \mA$ in the previous definition, then we are under the hypotheses of Subsection~\ref{subsec:structure_smooth}. Note that in this case $\mH = \emptyset$ and Definitions~\ref{def:sumsets_reg_smooth} and \ref{def:sumsets_reg_1psing} coincide for such a set $\mA$.
\item\label{rem:defsigma_2} 
For all $s\geq \sigma(\mA)$, one has that $(s+1)\mA \setminus s\mA = \Delta_{s+1,e} \setminus \Delta_{s,e}$. 
\end{enumerate}
\end{remark}

The following lemma characterizes when we have the equality $\Delta_{s,e} + \{\mathbf{0},D\bfeps_1',\ldots,D\bfeps_d'\} = \Delta_{s+1,e}$.

\begin{lemma} \label{lem:dependedimcone} 
$\Delta_{s,e} +\{\mathbf{0},D\bfeps_1',\ldots,D\bfeps_d'\} = \Delta_{s+1,e}$ \iff{} $s\geq d- \frac{d+e-1}{D}$. 
\end{lemma}
\begin{proof}
The inclusion $\Delta_{s,e} + \{\mathbf{0}, D\bfeps_1',\ldots,D\bfeps_d'\} \subset \Delta_{s+1,e}$ holds for every $s \in \N$, so let us prove that the reverse inclusion holds if and only if $s \geq d - \frac{d+e-1}{D}$.

If $s < d - \frac{d+e-1}{D}$, then $s D + 1 \leq d (D-1)-e+1$, so $sD+e\leq d(D-1)$ and, if we write $s D + e = q (D-1) + r$ with $q, r \in \N$ and $0 \leq r < D-1$, then either $q < d$ or $q = d$ and $r = 0$. Hence, we consider $\bfy = (D-1) (\bfeps_1' + \cdots + \bfeps_q') + r \bfeps_{q+1}' \in \N^d$. We have that $|\bfy|= sD + e$ and all its nonzero entries are $< D$. Thus, $\bfy \in \Delta_{s+1,e}$ and $\bfy \notin \Delta_{s,e} + \{\mathbf{0}, D\bfeps_1',\ldots,D\bfeps_d'\}$.

Assume now that $s \geq d - \frac{d+e-1}{D}$ and let us prove that $\Delta_{s,e} + \{\mathbf{0}, D\bfeps_1',\ldots,D\bfeps_d'\} = \Delta_{s+1,e}$. Take $\bfy \in \Delta_{s+1,e}$, if  $\bfy \in \Delta_{s,e}$, then $\bfy = \bfy + \mathbf{0}$. If $\bfy
\notin \Delta_{s,e}$, then $s D +e \leq  |\bfy| \leq (s+1) D$. Since $|\bfy| \geq sD+e \geq d(D-1)$, there exists $j \in \{1,\ldots,d\}$ such that $y_j \geq D$. Then $\bfy = (\bfy - D \bfeps_j') + D \bfeps_j' \in \Delta_{s,e} + \{\mathbf{0}, D\bfeps_1',\ldots,D\bfeps_d'\}$.
\end{proof}

When $D = 2$, we have that $\mH = \emptyset$ and $s\mA= \Delta_{s,e}$ for all $s \in \N$. As a consequence,  
$\sigma(\mA) =  \lceil d - \frac{d+2-1}{2} \rceil =  \lceil \frac{d-1}{2} \rceil$.
Whenever $D \geq 3$ we have the following generalization of Theorem \ref{thm:sigma_smooth}:

\begin{theorem} \label{thm:sigma_1psing_general}
Let $\mA = \{\bfa_0,\ldots,\bfa_n\} \subset \N^d$, $n\geq d \geq 1$, $D \geq 3$ be a set such that
\[\{\mathbf{0}\} \cup \{(D-1)\bfeps_i'+\bfeps_j' \mid 1\leq i,j\leq d\} \cup \{(D-e)\bfeps_i' \mid 1\leq i \leq d\} \subset \mA,\] for some $1\leq e \leq D$ such that $e$ divides $|\bfa_k|$ for all $k=0,\ldots,n$.
Then, the sumsets regularity of $\mA$ satisfies that
\[ d - \frac{d-e+1}{D} \leq \sigma(\mA) \leq \frac{D}{e} \left[ (d-1)(D-2) +\frac{D}{e}-2 \right].\]
\end{theorem}
\begin{proof}By Lemma \ref{lem:dependedimcone} we have that $\Delta_{s,e} +\{\mathbf{0},D\bfeps_1',\ldots,D\bfeps_d'\} = \Delta_{s+1,e}$ \iff{} $s\geq d- \frac{d+e-1}{D}$. Let $s' := \frac{D}{e} \left[ (d-1)(D-2) +\frac{D}{e}-2 \right]$, by Proposition \ref{prop:structure_sumsets} it follows that $\mH \subset \Delta_{s'}$ and that $s \mA = \Delta_{s,e} \setminus \mH$ for all $s \geq s'$. Since $d- \frac{d+e-1}{D} \leq s'$, we have that $\sigma(\mA) \leq s'$ and the result follows.
\end{proof}

\begin{remark}
The upper bound on $\sigma(\mA)$ in Theorem \ref{thm:sigma_1psing_general} is not sharp in general, as computational experiments suggest.
In the case $e=D$, it is sharp. This follows from Example \ref{ex:sigma_smooth}.
When $1<e<D$, the family \[\mA = \{\mathbf{0}\} \cup \{(D-1)\bfeps_i'+\bfeps_j' \mid 1\leq i,j\leq d\} \cup \{(D-e)\bfeps_i', \, e\bfeps_i' \mid 1\leq i \leq d\} \] satisfies $\sigma(\mA) \geq (\frac{D}{e}-2) d$. Indeed, the element $\bfy := (D-2e, \ldots, D-2e)$  can only be written as a combination of the elements of $\mA$ as $\bfy := \sum_{i = 1}^d \frac{D-2e}{e} (e \bfeps_i')$, hence $\bfy \in s \mA \setminus (s-1) \mA$ for $s = (\frac{D}{e}-2) d$. 
\end{remark}

\begin{example} \label{ex:1sing2} Consider the set \[ \mA = \{(0,0), (4,0), (0,4), (3,1), (1,3), (2,0), (0,2)\} \subset \N^2; \]
of Example \ref{ex:1sing}. Since $d = 2$, $D = 4$ and $e = 2$, then Theorem \ref{thm:sigma_1psing_general} ensures that 
$2 - \frac{3}{4} \leq \sigma(\mA) \leq  4$.
Nevertheless, in Example \ref{ex:1sing} we saw that 
\begin{itemize}
\item $\mH = \{(1,1)\} \subset \Delta_{1,2}$,
\item $1 \mA \neq \Delta_{1,2} \setminus \mH$, and 
\item $s \mA = \Delta_{s,2} \setminus \mH$ for all $s \geq 2$.
\end{itemize}
Moreover, $\Delta_{s,2} +\{(0,0),(4,0),(0,4)\} = \Delta_{s+1,2}$ \iff{} $s\geq 2 - \frac{3}{4}$, by Lemma~\ref{lem:dependedimcone}.
Hence, we obtain that $\sigma(\mA) = {\rm max}(1,2,\lceil 2 - \frac{3}{4}\rceil) = 2.$
\end{example}

We finish this section by showing that the behavior of Proposition \ref{prop:structure_sumsets} (i.e., $s\mA = \Delta_{s,e} \setminus \mH$ for all $s\gg 0$) characterizes the sets $\mA \subset \N^d$ of the form \eqref{eq:A_sup1singp} and, in turn, characterizes simplicial projective toric varieties with a single  singular point.

\begin{theorem} \label{thm:strThmCharacts1ptosing}
Let $n\geq d \geq 1$, $D\geq 2$, and $\mA = \{\bfa_0,\bfa_1,\ldots,\bfa_n\} \subset \N^d$ be a finite set, $\bfa_i = (a_{i1},\ldots,a_{id})$, such that $\{\mathbf{0},D\bfeps_1',\ldots,D\bfeps_d'\} \subset \mA$ and $|\bfa_i| \leq D$ for all $i \in \{0,\ldots,n\}$. 
The following conditions are equivalent:
\begin{enumerate}[(a)]
\item\label{thm:strThmCharacts1ptosing_1}  
$\mX_\uA \subset \Pn{n}_\k$, has a single singular point, and it is $(1:0:\dots:0)$.
\item\label{thm:strThmCharacts1ptosing_2} 
$\{ (D-1)\bfeps_i + \bfeps_j \mid  1\leq i,j\leq d \} \cup \{ e\bfeps_0 + (D-e) \bfeps_j \mid 1 \leq j \leq d\} \subset \uA$, where $\{\bfeps_0,\bfeps_1,\ldots,\bfeps_d\}$ is the canonical basis of $\N^{d+1}$,
$1\leq e \leq D$ is a divisor of $D$ that divides $a_{i0}$ for all $i\in \{0,\ldots,n\}$, and if $e=1$ then there exists $j\in \{1,\ldots,d\}$ such that $(D-1)\bfeps_0 + \bfeps_j \notin \uA$. 
\item\label{thm:strThmCharacts1ptosing_c} 
$\{\mathbf{0}\} \cup \{ (D-1)\bfeps_i' + \bfeps_j' \mid 1\leq i,j \leq d\} \cup \{ (D-e) \bfeps_i \mid 1\leq i \leq d \} \subset \mA$, where $1\leq e \leq D$ is a divisor of $D$ that divides $a_{i1}+\dots+a_{id}$ for all $i\in \{0,\ldots,n\}$, and if $e=1$ then there exists $j\in \{1,\ldots,d\}$ such that $\bfeps_j' \notin \mA$. 
\item\label{thm:strThmCharacts1ptosing_3} 
There exist a finite set $\mH \subset \N_e^d$, with $\mH\neq \emptyset$ if $e=1$, and a number $s_0 \in \N$ such that $s\mA = \Delta_{s,e}\setminus \mH$ for all $s \geq s_0$.
\end{enumerate}
\end{theorem}

\begin{proof}
The equivalence \ref{thm:strThmCharacts1ptosing_1} $\Leftrightarrow$ \ref{thm:strThmCharacts1ptosing_2} is Proposition~\ref{prop:charact_1singularpt},
\ref{thm:strThmCharacts1ptosing_2} $\Leftrightarrow$ \ref{thm:strThmCharacts1ptosing_c} is direct from the construction of $\uA$, 
and the implication \ref{thm:strThmCharacts1ptosing_c} $\Rightarrow$ \ref{thm:strThmCharacts1ptosing_3} is Proposition~\ref{prop:structure_sumsets}. 
Let us prove \ref{thm:strThmCharacts1ptosing_3} $\Rightarrow$ \ref{thm:strThmCharacts1ptosing_c}. 
Take $s \geq s_0$ such that $\mH \subset \Delta_{s-1,e}$, and fix $i,j$ with $1\leq i,j \leq d$. Since $(sD-1)\bfeps_i'+\bfeps_j' \in s\mA$, then $(D-1)\bfeps_i'+\bfeps_j' \in \mA$. Moreover, since $(sD-e)\bfeps_i' \in s\mA$, then $(D-e)\bfeps_i' \in \mA$.
If $e=1$, $\mH \neq \emptyset$ by hypothesis, so $\la \mA \ra \neq \N^d$. Hence, one has that there exists $j\in \{1,\ldots,d\}$ such that $\bfeps_j' \notin \mA$. Thus, we have proved \ref{thm:strThmCharacts1ptosing_c}. 
\end{proof}

\section{Castelnuovo-Mumford regularity and sumsets regularity}\label{sec3}

Let $\mA = \{\bfa_0,\ldots,\bfa_n\} \subset \N^d$, $n \geq d\geq 1$, with $\bfa_i = (a_{i1},\ldots,a_{id})$  and $|\bfa_i| \leq D$ for $i\in \{0,\ldots,n\}$, and assume that $\bfa_0 = \mathbf{0}$ and $\bfa_i = D\bfeps_i'$, $i\in \{1,\ldots,d\}$, for some $D \in \N$, $D\geq 2$. 
Let $\mS_\uA \subset \N^{d+1}$ be the affine semigroup generated by $\uA$, $\mS_\uA = \la \uA \ra$. By hypothesis, $\mS_\uA$ is simplicial, and the set of extremal rays of the rational cone spanned by $\uA$ is $\mathcal{E} := \{D\bfeps_0,\ldots,D\bfeps_d\}$. Let $\mX_\uA \subset \Pn{n}_\k$ be the projective toric variety determined by $\uA$; the coordinate ring of $\mX_\uA$ is isomorphic (as a graded $\kxo$-module) to the semigroup algebra $\k[\mS_\uA]$. \newline

For all $\bfs\in \mS_\uA$, consider the abstract simplicial complex $T_\bfs$ defined by 
\[T_\bfs \coloneq  \left\{ \mF \subset \mathcal{E} : \bfs - \sum_{\bfe \in \mF} \bfe \in \mS_\uA \right\} \,.\]
By \cite[Thm.~19]{Briales2000}, the Castelnuovo-Mumford regularity of $\k[\mX_\uA]$ is given by
\begin{equation}
\reg{\k[\mX_\uA]} = \max\left\{ \frac{|\bfs|}{D}- (i+1) : \bfs \in \mS_\uA \text{, and } \dim{\tilde{H}_i (T_\bfs)} \neq 0 \right\} .
\label{eq:CMreg_Briales}
\end{equation}
Note that this formula also comes from the short resolution of $\k[\mX_\uA]$ studied in \cite{GGG2025shortres}. \newline

In this section, we relate the Castelnuovo-Mumford regularity of $\k[\mX_\uA]$, $\reg{\k[\mX_\uA]}$, to the sumsets regularity of $\mA$, $\sigma(\mA)$, when $\mX_\uA$ is either smooth or has a single singular point. Moreover, we prove that the Eisenbud-Goto bound  holds for simplicial projective toric varieties of dimension $d \geq 3$ with a single singular point.

\subsection{The smooth case}

\begin{theorem}\label{thm:reg_sigma_smooth}
Let $\mA = \{\bfa_0,\ldots,\bfa_n\} \subset \N^d$, $n\geq d \geq 1$, be a set such that \[\{\mathbf{0}\} \cup \{(D-1)\bfeps_i'+\bfeps_j' \mid 1\leq i,j\leq d\} \cup \{\bfeps_i' \mid 1\leq i \leq d\} \cup \{(D-1)\bfeps_i' \mid 1\leq i \leq d\} \subset \mA, \]  
and consider $\mX_\uA \subset \Pn{n}_\k$ the smooth simplicial projective toric variety determined by $\uA$.
Then, 
\[\reg{\k[\mX_\uA]} = \sigma(\mA) \, .\]
\end{theorem}

\begin{lemma}\label{lem:reg_sigma_smooth}
With the hypotheses of Theorem~\ref{thm:reg_sigma_smooth}, let $\bfy \in \mS_\uA$ such that $|\bfy| = (\sigma+\ell)D$ for some $\ell \geq 1$. Then,
\[\widetilde{H}_i(T_\bfy) = 0 \text{, for all } -1\leq i \leq \min(\ell-2,d) \, .\]
\end{lemma}

\begin{proof}
Let $\pi: \N^{d+1} \rightarrow \N^d$ be the projection given by $\pi(z_0,z_1,\ldots,z_d) = (z_1,\ldots,z_d)$. Denote $s = \sigma+\ell$.
Denote by $\mathcal{E}_\bfy$ the vertex set of the simplicial complex $T_\bfy$, i.e., \[ \mathcal{E}_\bfy  = \{D \bfeps_{j} \, \vert \,
\bfy - D \bfeps_{j} \in \mS_\uA,\, 0 \leq j \leq d\}.\]

Let us prove that $\mE_\bfy \not= \emptyset$.  By \eqref{eq:relationsumsets}, we have that $\pi(\bfy) \in s \mA$.
Since $s\geq \sigma+1$, then $(s-1)\mA = \Delta_{s-1}, s \mA = \Delta_s$ and $\Delta_{s-1} + \{\mathbf{0}, D\bfeps_1',\ldots,D\bfeps_d'\} = \Delta_{s}$. Hence $(s-1)\mA + \{\mathbf{0}, D\bfeps_1',\ldots,D\bfeps_d'\} = s \mA$. If $\pi(\bfy) \in (s-1) \mA$, then $\bfy - D \bfeps_0 \in \mS_\uA$ and $D \bfeps_0 \in \mE_\bfy$. Otherwise,  $\bfz + D \bfeps_i' = \pi(\bfy)$ for some $\bfz \in (s-1)\mA$ and $i \in \{1,\ldots,d\}$. Again by \eqref{eq:relationsumsets}, we have that $ D \bfeps_i \in \mE_\bfy$.

{\it Claim:} Every subset of size at most $\min(\ell,d+1)$ of $\mathcal{E}_\bfy$ is a face of $T_\bfy$.

{\it Proof of the claim:} Take $D\bfeps_{j_1},\ldots,D\bfeps_{j_r} \in \mathcal{E}_\bfy$ with $r \leq \min(\ell,d+1)$. For all $1\leq k \leq r$, since $D\bfeps_{j_k} \in \mathcal{E}_\bfy$, then $y_{j_k} \geq D$. Hence, $\bfy - \sum_{k=1}^r D \bfeps_{j_k} \in \N^{d+1}$. Since $|\bfy-\sum_{k=1}^r D\bfeps_{j_k}| = (\sigma+\ell-r)D$ and $\sigma+\ell-r \geq \sigma$, then $\bfy-\sum_{k=1}^r D\bfeps_{j_k} \in \mS_\uA$, by the definition of $\sigma$. Therefore, $\{D\bfeps_{j_1},\ldots,D\bfeps_{j_r}\}$ is a face of $T_\bfy$.

If $|\mathcal{E}_\bfy| \leq \ell$ or $\ell \geq d+2$, then $T_\bfy$ is a full simplex, by the claim, and hence it is acyclic.
Assume $d+1 \geq |\mathcal{E}_\bfy| > \ell$. Then, by the claim, every subset of $\ell$ vertices of $\mathcal{E}_\bfy$ is a face of $T_\bfy$.
Therefore, 
\[
  \widetilde{H}_i(T_\bfy) \;=\; 0 \quad \text{for all } -1 \leq i\leq \ell-2.
\]
This follows from the following facts: (a) the full $(\ell-1)$-skeleton of the simplex is acyclic in dimensions $<\ell-1$, and (b) adding faces of dimension $> \ell-1$ does not affect homologies in dimension $i$ for $i < \ell-1$ (see, e.g., \cite{Hatcher}).
\end{proof}

\begin{proof}[Proof of Theorem~\ref{thm:reg_sigma_smooth}]
Set $\sigma = \sigma(\mA)$, and let us first prove that $\reg{\k[\mX_\uA]} \leq \sigma$. Take an element $\bfy \in \mS_\uA$.
\begin{itemize}
    \item If $|\bfy| \leq \sigma D$, then $\frac{|\bfy|}{D} -(i+1) \leq \sigma$ for all $i\geq -1$.
    \item If $|\bfy| = (\sigma + \ell)D$ for some $1\leq \ell \leq d+1$, then $\widetilde{H}_i(T_\bfy) = 0$ for all $-1\leq i \leq \ell-2$, by Lemma~\ref{lem:reg_sigma_smooth}, and note that $\frac{|\bfy|}{D}-(i+1) \leq \sigma$ for all $i > \ell-2$.
    \item If $|\bfy| \geq (\sigma+d+2)D$, then $\widetilde{H}_i(T_\bfy) = 0$ for all $-1\leq i \leq d$, by Lemma~\ref{lem:reg_sigma_smooth}.
\end{itemize}
Therefore, by \eqref{eq:CMreg_Briales}, it follows that $\reg{\k[\mX_\uA]} \leq \sigma$. To show the equality, we distinguish two cases.

If $(\sigma-1)\mA = \Delta_{\sigma-1}$, then there is an element $\bfy' \in \sigma \mA$ such that $\bfy' \notin (\sigma-1)\mA$ and $\bfy' - D\bfeps_i' \notin (\sigma-1)\mA$ for all $i\in \{1,\ldots,d\}$. Denote $\bfy = ( \sigma D - |\bfy'|, \bfy') \in \mS_\uA$. Then, one has that the simplicial complex $T_\bfy = \{ \emptyset \}$. Therefore, $\widetilde{H}_{-1} (T_\bfy) \neq 0$, and hence, $\reg{\k[\mX_\uA]} \geq \frac{|\bfy|}{D} = \sigma$.

Otherwise, consider an element $\bfz' \in \Delta_{\sigma -1 } \setminus (\sigma - 1)\mA$ and denote $\bfz = ( (\sigma-1)D - |\bfz'|, \bfz') \in \N^{d+1}$. One has that $|\bfz| = (\sigma-1)D$ and $\bfz \notin \mS_\uA$, and we take $\bfy = \bfz + \sum_{k=0}^d D\bfeps_k$. Observe that for every subset $\mF \subset \mE$ of size at most $d$, one has that $\bfy-\sum_{\bfe \in \mF} \bfe \in \mS_\uA$, since $|\bfy-\sum_{\bfe \in \mF} \bfe \in \mS_\uA| \geq \sigma D$, and $\bfy-\sum_{k=0}^d D\bfeps_k = \bfz \notin \mS_\uA$. Therefore, the simplicial complex $T_\bfy$ is the boundary of the $d$-simplex on the vertex set $\mE_\bfy = \{D\bfeps_0, \ldots,D\bfeps_d\}$, so $\widetilde{H}_i(T_\bfy) = 0$ for all $-1\leq i \leq d-2$ and $\widetilde{H}_{d-1}(T_\bfy) \neq 0$. Hence, by \eqref{eq:CMreg_Briales}, $\reg{\k[\mX_\uA]} \geq \frac{|\bfy|}{D}-d = \sigma$, which shows the desired equality.
\end{proof}

Combining Theorems~\ref{thm:sigma_smooth} and \ref{thm:reg_sigma_smooth}, we recover in Corollary \ref{cor:reg_smooth} the  bound for the Castelnuovo-Mumford regularity of $\k[\mX_\uA]$ by Herzog and Hibi \cite[Theorem~2.1]{Herzog2003}, which implies the Eisenbud-Goto bound, as shown in \cite[Corollary~2.2]{Herzog2003}.

\begin{corollary} \label{cor:reg_smooth}
If $\mX$ is a smooth simplicial projective toric variety, then 

\[ \reg{\k[\mX]} \leq \begin{cases} d(D-2) & \text{if } D \geq 3 ,\\ \lceil \frac{d}{2} \rceil & \text{if } D = 2. \end{cases}\]
\end{corollary}

\subsection{Varieties with at most one singular point}

\begin{theorem}\label{thm:reg_sigma_1psing}
Let $\mA = \{\bfa_0,\ldots,\bfa_n\} \subset \N^d$ be a set such that \[\{\mathbf{0}\} \cup \{(D-1)\bfeps_i'+\bfeps_j' \mid 1\leq i,j\leq d\} \cup \{(D-e)\bfeps_i' \mid 1\leq i \leq d\} \subset \mA,\] for some $1\leq e \leq D$ such that $e$ divides $|\bfa_k|$, and $|\bfa_k| \leq D$ for all $k=0,\ldots,n$.
Let $\mX_\uA \subset \Pn{n}_\k$ be the simplicial projective toric variety determined by $\uA$, which has at most  one singular point. Then,
\[\reg{\k[\mX_\uA]} \leq \sigma(\mA)+1 \, .\]
\end{theorem}

\begin{lemma}\label{lem:reg_sigma_1psing}
With the hypotheses of Theorem~\ref{thm:reg_sigma_1psing}, let $\bfy \in \mS_\uA$ such that $|\bfy| = (\sigma+1+\ell)D$ for some $\ell \geq 1$. Then,
\[\widetilde{H}_i(T_\bfy) = 0 \text{, for all } -1\leq i \leq \min(\ell-2,d) \, .\]
\end{lemma}

\begin{proof}
Let $\pi: \N^{d+1} \rightarrow \N^d$ be the projection given by $\pi(z_0,z_1,\ldots,z_d) = (z_1,\ldots,z_d)$, and consider $\pi(\bfy) \in (\ell+\sigma+1)\mA$. Set $s =|\bfy|/D = \ell+\sigma+1$. 
Consider the simplicial complex $T_\bfy$ and let $\mathcal{E}_\bfy$ be the vertex set of $T_\bfy$.

Let us prove that $\mE_\bfy \neq \emptyset$.  If $\pi(\bfy) \in (s-1) \mA$, then $\bfy - D \bfeps_0 \in \mS_\uA$ and $D \bfeps_0 \in \mE_\bfy$. If $\pi(\bfy) \notin (s-1) \mA$, we deduce that $|\pi(\bfy)| > (s-1)D$ because $(s-1) \mA = \Delta_{s-1,e} \setminus \mH$ and $\pi(\bfy) \notin \mH$. As $\Delta_{s-1,e} + \{\mathbf{0}, D\bfeps_1',\ldots,D\bfeps_d'\} = \Delta_{s,e}$, then $\bfz + D \bfeps_i' = \pi(\bfy)$ for some $\bfz \in \Delta_{s-1,e}$ and $i \in \{1,\ldots,d\}$. Since $|\bfz| = |\pi(\bfy)| - D > (s-2) D$, then $\bfz \notin \Delta_{s-2,e}$. Using that $s\geq \sigma+2$,  it follows that  $\bfz \notin \mH$ and then, $\bfz \in (s-1)\mA$. Thus, we conclude  that $ D \bfeps_i \in \mE_\bfy$.

\noindent\underline{Case 1:} Assume $D\bfeps_0 \notin \mE_\bfy$, i.e., $\bfy-D\bfeps_0 \notin \mS_\uA$. Then, $\pi(\bfy) \in s\mA \setminus (s-1)\mA$. Since $s=\ell+\sigma+1 \geq \sigma+1$, then $\pi(\bfy) \in \Delta_{s,e} \setminus \Delta_{s-1,e}$, by Remark~\ref{rem:defsigma}~\ref{rem:defsigma_2}, and hence, $|\pi(\bfy)| > (s-1)D$.

{\it Claim:} Every subset of size at most $\min(\ell,d+1)$ of $\mathcal{E}_\bfy$ is a face of $T_\bfy$.

{\it Proof of the claim:} Take $D\bfeps_{j_1},\ldots,D\bfeps_{j_r} \in \mathcal{E}_\bfy$ with $r \leq \min(\ell,d+1)$. For all $1\leq k \leq r$, since $D\bfeps_{j_k} \in \mathcal{E}_\bfy$, then $y_{j_k} \geq D$. Hence, $\bfy - \sum_{k=1}^r D \bfeps_{j_k} \in \N^{d+1}$, and its first coordinate is a multiple of $e$. We observe that 
\[(s-r)D \geq |\pi ( \bfy- \sum_{k=1}^r D\bfeps_{j_k} )| = |\pi(\bfy)|-rD > (s-1-r)D, \]
and since $s-1-r =  \sigma+\ell-r\geq \sigma$, by Remark~\ref{rem:defsigma}~\ref{rem:defsigma_2},
it follows that $\pi(\bfy-\sum_{k=1}^r D\bfeps_{j_k}) \in (s-r)\mA$. Therefore, $\bfy-\sum_{k=1}^r D\bfeps_{j_k} \in \mS_\uA$, and $\{D\bfeps_{j_1},\ldots,D\bfeps_{j_r}\}$ is a face of $T_\bfy$. 

From the Claim, it follows that $\widetilde{H}_i(T_\bfy) = 0$ for $-1\leq i \leq \min(\ell-2,d)$, as in the proof of Lemma~\ref{lem:reg_sigma_smooth}.

\noindent\underline{Case 2:} Assume $D\bfeps_0 \in \mE_\bfy$, i.e., $\bfy-D\bfeps_0 \in \mS_\uA$. Then, $\pi(\bfy) \in (s-1)\mA$, so $|\pi(\bfy)| \leq (s-1)D$. For all $m\in \N$, consider $\bfz_m = \bfy+m (D\bfeps_0)$. 

{\it Claim:} The faces of $T_\bfy$ and $T_{\bfz_m}$ coincide up to cardinality $\min(\ell,d+1)$.

{\it Proof of the claim:} We have to prove the following 
\[\forall \, 1\leq r \leq \min(\ell,d+1) , \quad \bfy-\sum_{k=1}^r D\bfeps_{j_k} \in \mS_\uA \Longleftrightarrow \bfy+m(D\bfeps_0)-\sum_{k=1}^r D\bfeps_{j_k} \in \mS_\uA \,.\]
Being $(\Rightarrow)$ straightforward, let us prove $(\Leftarrow)$. Set $\bfy' = \bfy-\sum_{k=1}^r D\bfeps_{j_k}$.
Suppose that $\bfy'+mD\bfeps_0 \in \mS_\uA$. Then, $\pi(\bfy'+mD\bfeps_0)= \pi(\bfy') \in (s+m-r)\mA$. To prove that $\bfy' \in \mS_\uA$, let us show that $\pi(\bfy') \in (s-r)\mA$. We have that $\pi(\bfy') \in \la \mA\ra$, $|\pi(\bfy')| = |\pi(\bfy)|-rD \leq (s-r)D$ and $s-r\geq \ell+\sigma+1-\ell = \sigma+1$, and hence $\pi(\bfy') \in (s-r)\mA$, by the definition of $\sigma$. Therefore, we have proved that $\bfy-\sum_{k=1}^r D\bfeps_{j_k} \in \mS_\uA$.

From the claim, it follows that $\widetilde{H}_i(T_\bfy) = \widetilde{H}_i(T_{\bfz_m})$ for all $-1\leq i \leq \min(\ell-2,d)$. If there exists $i$, $-1\leq i \leq \min(\ell-2,d)$ such that $\widetilde{H}_i (T_\bfy) \neq 0$, then by \eqref{eq:CMreg_Briales} one has that 
\[\reg{\k[\mX_\uA]} \geq \max_{m \in \N} \{s+m-(i+1) \mid \widetilde{H}_i(T_\bfy) \neq 0\} = \infty ,\]
which is absurd. Then, one has that $\widetilde{H}_i(T_\bfy) = 0$ for all $-1\leq i \leq \min(\ell-2,d)$.
\end{proof}

\begin{proof}[Proof of Theorem~\ref{thm:reg_sigma_1psing}]
Set $\sigma = \sigma(\mA)$, and let us prove that $\reg{\k[\mX_\uA]} \leq \sigma+1$. Take an element $\bfy \in \mS_\uA$.
\begin{itemize}
    \item If $|\bfy| \leq (\sigma+1) D$, then $\frac{|\bfy|}{D} -(i+1) \leq \sigma+1$ for all $i\geq -1$.
    \item If $|\bfy| = (\sigma+1+\ell)D$ for some $1\leq \ell \leq d+1$, then $\widetilde{H}_i(T_\bfy) = 0$ for all $-1\leq i \leq \ell-2$, by Lemma~\ref{lem:reg_sigma_smooth}, and note that $\frac{\bfy}{D}-(i+1) \leq \sigma+1$ for all $i > \ell-2$.
    \item If $|\bfy| \geq (\sigma+d+3)D$, then $\widetilde{H}_i(T_\bfy) = 0$ for all $-1\leq i \leq d$, by Lemma~\ref{lem:reg_sigma_smooth}.
\end{itemize}
Therefore, by \eqref{eq:CMreg_Briales}, it follows that $\reg{\k[\mX_\uA]} \leq \sigma+1$.
\end{proof}

\begin{corollary} \label{cor:reg_1sing}
If $\mX$ is a  simplicial projective toric variety with exactly one singular point, then 
\[ \reg{\k[\mX]} \leq \begin{cases} \frac{D}{e} \left[ (d-1)(D-2) +\frac{D}{e}-2 \right] + 1 & \text{if } D \geq 3 ,\\ \lceil \frac{d-1}{2} \rceil & \text{if } D = 2. \end{cases}\]
\end{corollary}
\begin{proof}
When $D \geq 3$, the result follows directly from Theorems \ref{thm:reg_sigma_1psing} and \ref{thm:sigma_1psing_general}. 
For $D = 2$, we have that $e = 2$ and $\mX = \mX_{\uA}$, where
 $\uA = \{2 \bfeps_0\} \cup \{\bfeps_i+\bfeps_j \mid 1\leq i \leq j\leq d\} \subset \N^{d+1}$.
If we set $\underline{\mB} := \{\bfeps_i'+\bfeps_j' \mid 1\leq i \leq j\leq d\} \subset \N^{d}$, we have that $I_\uA = I_{\underline{\mB}} .\k[x_0,\ldots,x_n]$ and, thus, $\reg{\k[\mX_{\uA}]} = \reg{\k[\mX_{\underline{\mB}}]}$, as stated in Remark~\ref{rem:sup1psing}~\ref{rem:sup1psing_ii}. Moreover,  $X_{\underline{\mB}} \subset \P_\k^{n-1}$ is a smooth simplicial projective toric variety and, by Corollary \ref{cor:reg_smooth}, $\reg{\k[\mX_{\underline{\mB}}]} \leq \lceil \frac{d-1}{2}\rceil$.
\end{proof}

\begin{example}\label{ex:1sing3} If we consider the set $\mA$ of Examples \ref{ex:1sing} and \ref{ex:1sing2}, by Corollary \ref{cor:reg_1sing}, we have that $\reg{\k[\mX_\uA]} \leq 5$. Nevertheless, since we know that $\sigma(\mA) = 2$, then Theorem \ref{thm:reg_sigma_1psing} gives $\reg{\k[\mX_\uA]} \leq 3$.  We also observe that, 
\[ \uA = \{\bfa_0, \bfa_1, \bfa_2, \bfa_3,  \bfa_4, \bfa_5, \bfa_6\} \subset \N^3, \]
with $\bfa_i = 4 \bfeps_i$ for $i = 0,1,2,\, \bfa_3 = (0,3,1), \bfa_4 = (0,1,3), \bfa_5 = (2,2,0)$ and $\bfa_6 = (2,0,2)$. 
In other words, \[ \uA = \{(y_0,y_1,y_2) \, \vert \, y_0 + y_1 +\ y_2 = 4,\, y_0 \text{ is even, and }  (y_1,y_2) \notin \{(1,1), (2,2)\} \}.\] 
Since $s \mA = \Delta_{s,2} \setminus \{(1,1)\}$ for all $s \geq 2$,  then 
\[s \uA = \{(y_0,y_1,y_2) \, \vert \, y_0 + y_1 +\ y_2 = 4s, y_0 \text{ is even, and } \, (y_1,y_2) \neq (1,1)\}.\]
If one considers $\bfy = (4,2,2)$, then one has that $\bfy = \bfa_5 + \bfa_6 \in \mS_\uA$, but $\bfy - D \bfeps_1,  \bfy - D \bfeps_2 \notin \N^3$ and $\bfy - D \bfeps_0 = (0,2,2) \notin \uA$. Hence, the simplicial complex $T_{\bfy}$ is empty, which implies that $\dim{\tilde{H}_{-1} (T_\bfy)} \neq 0$ and, by  \eqref{eq:CMreg_Briales}, $\reg{\k[\mX_\uA]} \geq 2$.

Indeed, a direct computation with the Sage package {\tt Shortres} \cite{github_shortres} confirms that $\reg{\k[\mX_\uA]} = 2$.
\end{example}

We include another example with $\reg{\k[\mX_\uA]} = \sigma(\mA) + 1$.

\begin{example}\label{ex:1sing2_2} Consider the set \[ \mA = \{(y_1,y_2) \in \N^2 \, \vert \, y_1 + y_2 \text{ is even},\, y_1 + y_2 \leq 6\} \setminus \{(2,4), (3,3)\} \subset \N^2; \]
this set satisfies the conditions in \eqref{eq:A_sup1singp} with $D = 6, d = e = 2$.
We observe (see Figure \ref{fig:ex1sing2}) that $\mA = \Delta_{1,2} \setminus \{(2,4), (3,3)\},\, 2 \mA =  \Delta_{2,2} \setminus \{(3,9)\}$, and
$s \mA = \Delta_{s,2}$ for all $s \geq 3$. Therefore, $\mH = \emptyset$. 

\begin{figure}
\includegraphics[scale = .9]{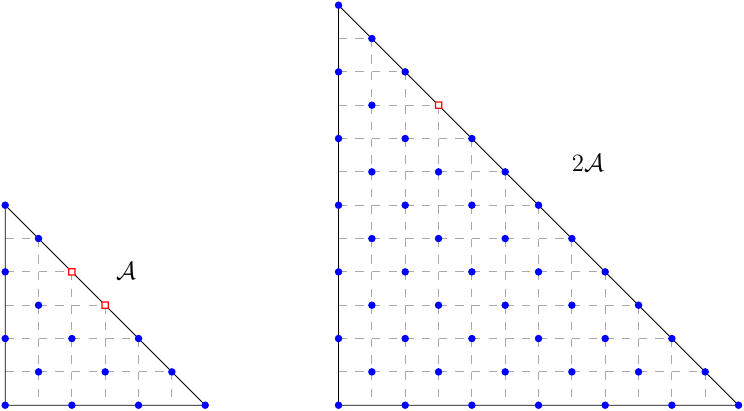}
\caption{For $\mA$ as in Example \ref{ex:1sing2_2}, filled circles represent elements in $\mA$ and $2 \mA$, respectively; while empty squares correspond to elements in $\Delta_{1,2} \setminus \mA$ and $\Delta_{2,2} \setminus 2\mA$, respectively.} \label{fig:ex1sing2}
\end{figure}

Moreover, by Lemma~\ref{lem:dependedimcone}, $\Delta_{s,2} +\{(0,0),(6,0),(0,6)\} = \Delta_{s+1,2}$ \iff{} $s\geq 2 - \frac{1}{2}$. Hence, $\sigma(\mA) = {\rm max}(2,\lceil 2 - \frac{1}{2}\rceil) = 2$ and, by Theorem \ref{thm:reg_sigma_1psing},
$\reg{\k[\mX_\uA]} \leq 3$. 

Consider now $\bfy = (6,9,15) \in \N^3$. We have that:
\begin{itemize} 
\item $\bfy - 6 (\bfeps_0 + \bfeps_1) = (0,3,15) \in \mS_\uA$ because $(3,15) \in 3 \mA$,
\item $\bfy - 6 (\bfeps_0 + \bfeps_2) = (0,9,9)\in \mS_\uA$ because $(9,9) \in 3 \mA$,
\item $\bfy - 6 (\bfeps_1 + \bfeps_2) = (6,3,9)\in \mS_\uA$ because $(3,9) \in 3 \mA$, and
\item $\bfy - 6 (\bfeps_0 + \bfeps_1 + \bfeps_2) = (0,3,9) \notin \mS_\uA$ because $(3,9) \notin 2\mA$.
\end{itemize} 
 Hence the simplicial complex $T_\bfy$ is a hollow triangle. Thus, ${\rm dim}(\widetilde{H}_1(T_\bfy)) = 1$ and, by \eqref{eq:CMreg_Briales}, we have that $\reg{\k[\mX_\uA]} \geq \frac{6+9+15}{6} - 2 = 3.$
\end{example}

\subsection{The Eisenbud-Goto bound}

\begin{theorem} \label{thm:EG_surfaces_1psing}
Let $\mX \subset \Pn{n}_\k$ be a simplicial projective toric variety with exactly one singular point. If the dimension of $\mX$ is $d \geq 3$ (i.e., $\dim(\k[\mX]) = d+1 \geq 4$), then the Eisenbud-Goto bound holds for $\k[\mX]$, i.e.,
\[ \reg{\k[\mX]} \leq \deg(\mX) - n+d \, . \]
\end{theorem}

Before proving Theorem~\ref{thm:EG_surfaces_1psing}, we compute the degree of $\mX$ in Proposition~\ref{prop:deg_1psing}, estimate the codimension of $\mX$ in Lemma~\ref{lem:bound_sizeA} and show in Lemma~\ref{lem:EG_e_neq_D} an inequality that will be needed in the proof of the theorem.

The degree of simplicial projective toric varieties can be computed using the following result.

\begin{lemma}[{\cite[Thm.~2.13, Thm.~4.5]{OCarroll2014}}] \label{lem:deg_XA}
Let $\mX_\uA \subset \Pn{n}_\k$ be a simplicial projective toric variety, where $\uA = \{\ubfa_0,\ldots,\ubfa_n\} \subset \N^{d+1}$, $|\ubfa_i| = D \in \Z_{>0}$ for all $i=0,\ldots,n$ and $\{D\bfeps_0,\ldots,D\bfeps_d\} \subset \uA$. 
The degree of $\mX_\uA$ can be computed as
\[
\deg(\mX_\uA) = \frac{(d+1)! \cdot {\rm vol} \left( {\rm conv} \left( \uA \cup \{\mathbf{0} \} \right) \right)}{\theta_{d+1}} = \dfrac{D^{d+1}}{\theta_{d+1}} \, ,
\]
where
${\rm vol} \left( {\rm conv} \left( \uA \cup \{\mathbf{0} \} \right) \right)$ denotes the volume of the convex hull of $\uA \cup \{\mathbf{0} \} \subset \R^{d+1}$, and
$\theta_{d+1}$ is the greatest common divisor of the $(d+1)\times(d+1)$ minors of the $(d+1)\times (n+1)$ matrix $M$, whose columns are the vectors $\ubfa_0,\ldots,\ubfa_n$.
\end{lemma}

\begin{proposition} \label{prop:deg_1psing}
Suppose that 
\[\{(D-1)\bfeps_i+\bfeps_j \mid 1\leq i,j \leq d\} \cup \{ e\bfeps_0 + (D-e) \bfeps_j \mid 1\leq j \leq d\}\subset \uA \, ,\] 
where $1\leq e\leq D$ is a divisor of $D$ that divides $a_{i0}$ for all $i=0,\ldots,n$, and let $\mX_\uA$ be the projective toric variety determined by $\uA$. Then, the degree of $\mX_\uA$ is $\deg(\mX_\uA) = \frac{D^d}{e}$.
\end{proposition}

\begin{proof}
Consider the matrix $M$ of size $(d+1)\times (n+1)$ whose columns are the elements of $\uA$. By Lemma~\ref{lem:deg_XA}, the degree of the toric variety $\mX_\uA$ is $\deg(\mX_\uA) = D^{d+1}/\theta_{d+1}$. Since the first row of $M$ is a multiple of $e$ and the sum of the entries in any column is $D$, then $D\cdot e$ divides $\theta_{d+1}$. Moreover, the determinant of the (upper triangular) matrix whose columns are $e\bfeps_0+(D-e)\bfeps_1$, $D\bfeps_1$ and $(D-1)\bfeps_i+\bfeps_{i+1}$, $1\leq i \leq d-1$, is $D \cdot e$. Hence, $\theta_{d+1} = D \cdot e$. Thus, $\deg(\mX_\uA) = \frac{D^{d+1}}{D \cdot e} = \frac{D^d}{e}$, by Lemma~\ref{lem:deg_XA}.
\end{proof}

Therefore, the Eisenbud-Goto bound for a simplicial projective toric variety of dimension $d$ with exactly one singular point, $\mX_\uA \subset \Pn{n}_\k$, parametrized as in Proposition~\ref{prop:charact_1singularpt}~\ref{prop:charact_1singularpt_2}, can be written as
\begin{equation}
\reg{\k[\mX_\uA]} \leq \frac{D^d}{e}-n+d = \frac{D^d}{e}-|\mA|+d+1 .
\label{eq:EGconj_sup1psing}
\end{equation}

\begin{lemma} \label{lem:bound_sizeA}
Let $\mA \subset \N_e^d$ such that $|\bfy| \leq D$ for all $\bfy\in \mA$, and $1\leq e < D$. Then, 
\[|\mA| \leq \frac{\frac{D}{e}+d}{D+d} {D+d \choose d} \, .\]
\end{lemma}

\begin{proof}
Since $\mA \subset \N_e^d$ and $|\bfy| \leq D$ for all $\bfy\in \mA$, one has that 
\[|\mA| = \sum_{i=0}^{D/e} |\mA \cap \{\bfy \in \N^d : x_1+\dots+x_d = ie\}| \leq \sum_{i=0}^{D/e} {ie+d-1 \choose d-1} \, .\]
On the other hand, note that for all $0\leq i < D/e$,
\[e \cdot {ie+d-1 \choose d-1} \leq | \{\bfy \in \N^d : ie \leq |\bfy| < (i+1)e\}| \, .\]
Therefore, one has that 
\[e \cdot \sum_{i=0}^{D/e-1} {ie+d-1 \choose d-1} \leq |\{\bfy \in \N^d : 0 \leq |\bfy| < D\}| \, ,\]
and hence,
\[\begin{split}
|\mA| \leq \sum_{i=0}^{D/e} {ie+d-1 \choose d-1} &\leq \frac{1}{e} |\{\bfy \in \N^d : 0 \leq |\bfy| < D\}| + |\{\bfy \in \N^d : |\bfy| = D\}| \\
&= \frac{1}{e} {D+d \choose d} + \left( 1-\frac{1}{e} \right) {D+d-1 \choose d-1}  \\
&= \frac{\frac{D}{e}+d}{D+d} {D+d \choose d} \, .
\end{split}\]
\end{proof}

The last ingredients for the proof of Theorem~\ref{thm:EG_surfaces_1psing} are the following two lemmas.

\begin{lemma}[{\cite[Lemma~1.2]{Herzog2003}}] \label{lem:EG_e=D}
If $D\geq 3$ and $d\geq 3$, then 
\[(d-1)(D-2) \leq D^{d-1} - {D+d-1 \choose d-1} +d \, .\]
\end{lemma}

\begin{lemma} \label{lem:EG_e_neq_D}
For all $D\geq 3$, $d\geq 3$, and $1\leq e < D$, where $e$ is a divisor of $D$, one has that
\[\frac{D}{e} \left[ (d-1)(D-2)+\frac{D}{e}-2 \right] \leq \frac{D^d}{e}- \frac{\frac{D}{e}+d}{D+d} {D+d \choose d} +d .\]
\end{lemma}

\begin{proof}
The inequality is equivalent to 
\[D^d- \frac{D+de}{D+d} {D+d \choose d} \geq D \left[ (d-1)(D-2)+\frac{D}{e}-2 \right] -ed.\]
Thus, we must prove
\begin{equation}
D^d-  \frac{D+de}{D+d} \prod_{i=1}^d \left( 1 + \frac{D}{i} \right) \geq D \left[ (d-1)(D-2)+\frac{D}{e}-2 \right] -ed
\label{eq:lemma_EG}
\end{equation}
for $d\geq 3$, $D\geq 3$, and $1\leq e < D$. Fix $D\geq 3$. We prove \eqref{eq:lemma_EG} by induction on $d\geq 3$. 

Assume that $d=3$. Equation~\eqref{eq:lemma_EG} is equivalent to 
\begin{equation}
5eD^3 - (3e^2+15e+6)D^2 + (34e-9e^2)D+12e^2 \geq 0 \, .  
\label{eq:lemma_EG_d=3}
\end{equation}
Let us show that it holds for all $1\leq e < D$.
\begin{itemize}
    \item If $e=1$, \eqref{eq:lemma_EG_d=3} becomes $5D^3-24D^2+25D+12 \geq0$, which holds for all $D\geq 3$.
    \item If $2 \leq e < D/2$, then we have that 
    \[\begin{split}
     5 \left( \frac{D}{e} \right)^3 - &\left( 3+\frac{15}{e}+\frac{6}{e^2} \right) \left( \frac{D}{e} \right)^2  + \left( \frac{34}{e^2} - \frac{9}{e} \right) \frac{D}{e} + \frac{12}{e^2} \\
     &\geq 5 \left( \frac{D}{e} \right)^3 - \left( 3+\frac{15}{2}+\frac{6}{4} \right) \left( \frac{D}{e} \right)^2 - \frac{9}{2} \frac{D}{e} \\
     &= \frac{D}{e} \left[5 \left( \frac{D}{e} \right)^2 - 12 \frac{D}{e} - \frac{9}{2}  \right] \geq 0
    \end{split}\]
    for all $D/e \geq 3$, and hence \eqref{eq:lemma_EG_d=3} holds for all $2\leq e < D/2$. 
    \item If $D$ is even and $e=D/2$, then \eqref{eq:lemma_EG_d=3} is equivalent to $D^2(7D^2-39D+56) \geq 0$, which holds for all $D\geq 0$.
\end{itemize}

Take $d\geq 3$ and suppose that the inequality \eqref{eq:lemma_EG} is true for $d$. Inequality \eqref{eq:lemma_EG} for $d+1$ then follows from the computation below:

{
\allowdisplaybreaks
\begin{align*}
&{} D^{d+1} - \frac{D+de+e}{D+d+1} \prod_{i=1}^{d+1} \left( 1+\frac{D}{i} \right) \\
&= (D^{d+1}-D^d) - \frac{D+de+e}{D+d+1} \prod_{i=1}^{d+1} \left( 1+\frac{D}{i} \right) + \frac{D+de}{D+d} \prod_{i=1}^{d} \left( 1+\frac{D}{i} \right) \\ 
&\qquad + D^d- \frac{D+de} {D+d} \prod_{i=1}^d \left( 1+\frac{D}{i} \right) \\
&\geq D^d(D-1) + \left[ \frac{D+de}{D+d} - \frac{D+de+e}{D+d+1} \left( 1 + \frac{D}{d+1} \right) \right] \prod_{i=1}^{d} \left( 1+\frac{D}{i} \right) \\
&\qquad+ D \left[ (d-1)(D-2)+\frac{D}{e}-2 \right] -ed \\
&= D^d(D-1) + \left( \frac{D+de}{D+d} - \frac{D}{d+1} -e \right) \prod_{i=1}^{d} \left( 1+\frac{D}{i} \right)   \\
&\qquad + D \left[ (d-1)(D-2)+\frac{D}{e}-2 \right] -ed \\
&\geq D^d(D-1) + \left( 1 - \frac{D}{d+1} -e \right) D^d + D \left[ (d-1)(D-2)+\frac{D}{e}-2 \right] -ed \\
&=  D^d \left( D- \frac{D}{d+1} -e \right) -D(D-2) +e  + D \left[ d(D-2)+\frac{D}{e}-2 \right] - e(d+1). \\
\end{align*}
}

To conclude the proof we just need to show that $D^d \left( D -\frac{D}{d+1} -e \right) -D(D-2)+e \geq 0$. Since $D\geq 3$, $d\geq 3$ and $1\leq e \leq D/2$, then one can bound
\[D-\frac{D}{d+1} -e \geq D-\frac{D}{4} -\frac{D}{2} = \frac{D}{4} .\]
Therefore, applying again that $d\geq 3$ one has that 
\[D^d \left( D-\frac{D}{d+1} -e \right) -D(D-2)+e \geq \frac{D^4}{4}-D^2+2D+1 \geq 0 \, , \quad\forall D\geq 0.\]

\end{proof}

\begin{proof}[Proof of Theorem~\ref{thm:EG_surfaces_1psing}]
Without loss of generality, we may assume that the only singular point of $\mX$ is $P_0 = (1:0:\dots:0)$. Then, there is a set $\mA \subset \N^d$ as in Subsection~\ref{subsec:structure_1psing} such that $\mX = \mX_\uA$, i.e., 
\[\{\mathbf{0}\} \cup \{D\bfeps_i' \mid 1\leq i \leq d \} \cup \{ (D-1)\bfeps_i'+\bfeps_j' \mid 1\leq i,j \leq d \} \cup \{(D-e) \bfeps_j' \mid 1\leq j \leq d\} \subset \mA,\]
where $|\bfa_i| \leq D$,
$1\leq e \leq D$ is a divisor of $D$ that divides $|\bfa_i|$ for all $i\in \{0,\ldots,n\}$; and if $e=1$, then there exists $j\in \{1,\ldots,d\}$ such that $\bfeps_j' \notin \mA$.

If $D = 2$, then $e=2$ and $|\mA| \leq {d+1 \choose 2}+1 $. By Corollary \ref{cor:reg_1sing}, we have that 
\[\reg{\k[\mX_\uA]} \leq \left\lceil \frac{d-1}{2} \right\rceil \leq 2^{d-1} - {d+1 \choose 2} +d \leq 2^{d-1} - |\mA| +d+1 \, ,\]
and hence the Eisenbud-Goto bound  \eqref{eq:EGconj_sup1psing} holds for $\k[\mX_\uA]$.

If $D \geq 3$ and $1\leq e  < D$, by Corollary \ref{cor:reg_1sing}, Lemma \ref{lem:EG_e_neq_D}, Lemma~\ref{lem:bound_sizeA}, and Proposition~\ref{prop:deg_1psing} one has that 
\[\begin{split} 
\reg{\k[\mX_\uA]} & \leq  \frac{D}{e} \left[ (d-1)(D-2)+\frac{D}{e}-2 \right] +1 \\
&\leq \frac{D^d}{e} - \frac{\frac{D}{e}+d}{D+d} {D+d \choose d} +d +1 \\
&\leq \frac{D^d}{e} - |\mA| +d+1 = \deg(\mX_\uA) -|\mA|+d+1 \, ,
\end{split}\]
and hence the Eisenbud-Goto bound  \eqref{eq:EGconj_sup1psing} holds for $\k[\mX_\uA]$.

Finally, if $D \geq 3$ and $e=D$, one has that $|\mA| \leq {D+d-1 \choose d-1} +1$. By Corollary \ref{cor:reg_1sing}, Lemma~\ref{lem:EG_e=D}, and Proposition~\ref{prop:deg_1psing} one has that 
\[
\reg{\k[\mX_\uA]} \leq  (d-1)(D-2) \leq D^{d-1} - {D+d-1 \choose d-1} +d \leq D^{d-1}-|\mA|+d+1 \, ,
\]
and hence the Eisenbud-Goto bound  \eqref{eq:EGconj_sup1psing} holds for $\k[\mX_\uA]$.
\end{proof}

\section*{Conclusions}

Given a projective toric variety $\mX_\uA \subset \P_\k^n$, we investigated the Castelnuovo--Mumford regularity of $\k[\mX_\uA]$ via the sumsets of $\mA \subset \N^{d}$. For simplicial toric varieties with at most one singular point, we introduced the notion of sumsets regularity $\sigma(\mA)$, which measures the point from which the sumsets of $\mA$ exhibit predictable behavior. In Theorems \ref{thm:sigma_smooth} and \ref{thm:sigma_1psing_general} we established upper bounds for $\sigma(\mA)$. Furthermore, we proved that $\reg{\k[\mX_\uA]} = \sigma(\mA)$ in the smooth case (Theorem \ref{thm:reg_sigma_smooth}), while $\reg{\k[\mX_\uA]} \leq \sigma(\mA)+1$ when $\mX_\uA$ has exactly one singular point (Theorem \ref{thm:reg_sigma_1psing}). These results yield explicit upper bounds for the Castelnuovo--Mumford regularity of $\k[\mX_\uA]$ in both settings. In the smooth case, we also gave an alternative proof of a result of Herzog and Hibi \cite{Herzog2003}. In the singular case, when $d \ge 3$, our results identify new families of non-smooth varieties that satisfy the Eisenbud--Goto bound.

A natural direction for future research is to investigate whether the techniques developed here can be used to prove that projective simplicial toric surfaces satisfy the Eisenbud--Goto bound. Partial results in this direction appear in \cite{MarioTesis}. Another interesting problem is whether these methods extend to simplicial, or even non-simplicial, toric varieties with more general singular loci.

\end{document}